\documentclass[11pt]{article}

\usepackage[margin=1in]{geometry}
\usepackage{amsmath,amssymb,amsthm,mathtools}
\usepackage{bm}
\usepackage{graphicx}
\usepackage[hidelinks]{hyperref}
\usepackage[nameinlink,noabbrev]{cleveref}
\usepackage{enumitem}
\usepackage[round,authoryear]{natbib}

\theoremstyle{plain}
\newtheorem{theorem}{Theorem}[section]
\newtheorem{lemma}[theorem]{Lemma}

\newtheorem{corollary}[theorem]{Corollary}

\theoremstyle{definition}
\newtheorem{definition}[theorem]{Definition}
\newtheorem{remark}[theorem]{Remark}

\title{Estimation of Riemannian Quantities from Noisy Data via Density Derivatives}

\author{
Junhao Chen\thanks{\texttt{jhochenn@gmail.com}, Independent Researcher}
\and
Ruowei Li\thanks{\texttt{ruoweili@nus.edu.sg}, Department of Statistics and Data Science, National University of Singapore}
\and
Zhigang Yao\thanks{\texttt{zhigang.yao@nus.edu.sg}; \texttt{zhigang.yao@cmsa.fas.harvard.edu}. Department of Statistics and Data Science, National University of Singapore; Center of Mathematical Sciences and Applications, Harvard University}
}

\date{}

\begin{document}
\maketitle

\begin{abstract}%
    We study the recovery of geometric structure from data generated by convolving the uniform measure on a smooth compact submanifold $M\subset\mathbb{R}^D$ with ambient Gaussian noise. Our main result is that several fundamental Riemannian quantities of $M$, including tangent spaces, the intrinsic dimension, and the second fundamental form, are identifiable from derivatives of the noisy density.
    We first derive uniform small-noise expansions of the data density and its derivatives in a tubular neighborhood of $M$. These expansions show that, at the population level, tangent spaces can be recovered from the density Hessian with $O(\sigma^2)$ error, while the intrinsic dimension can be estimated consistently. We further construct estimators for the second fundamental form from density derivatives, obtaining $O(d(y,M)+\sigma)$ and $O(d(y,M)+\sigma^2)$ errors for hypersurfaces and submanifolds with arbitrary codimension.
    At the sample level, we estimate the density and its derivatives by kernel methods in the ambient space and plug them into the population constructions, yielding uniform nonparametric rates in the ambient dimension.
    Finally, we show that these density-based constructions admit a geometric interpretation through density-induced ambient metrics, linking the geometry of $M$ to ambient geodesic structure.
\end{abstract}


\section{Introduction} 
A common modeling paradigm for high-dimensional data in $\mathbb{R}^D$ posits an unknown $d$-dimensional smooth submanifold $M \subset\mathbb{R}^D$ with $d \ll D$. 
In much of the literature, however, the geometry of $M$ is imposed primarily as an a priori condition, rather than treated as a collection of geometric objects to be inferred from data. This viewpoint leads to a question of statistical identifiability: what geometric structure of $M$ is recoverable from noisy observations, under what noise conditions, and at what accuracy?

One line of research, often referred to as manifold estimation or manifold fitting, studies recovery of the underlying set $M$ from discrete observations. 
The aim is to construct a $d$-dimensional submanifold $\widehat M$ that is close to $M$, while satisfying essential regularity conditions, such as smoothness, bounded curvature, and a positive lower bound on the reach; see, e.g., \citet{Genovese2012manifold, Genovese2014Ridge, Fefferman2018FittingAP, Yao2019ManifoldFU, Sober2020, Fefferman2023LargeReach, yao2023manifold, Fefferman2023fittingmanifolddatapresence}. Beyond set estimation, one may also seek to infer regularity quantities such as the reach \citep{Aamari2019reach}.

In this paper, we study a complementary problem of recovering the intrinsic Riemannian structure of $M$.
For an embedded submanifold $M \subset \mathbb{R}^D$, two fundamental extrinsic primitives are the tangent spaces $\{T_xM\}_{x\in M}$ and the second fundamental form $\{\Pi_x\}_{x\in M}$. The tangent space determines the induced metric, while $\Pi_x$ describes the normal component of the ambient derivative; together, via the Gauss formula and the Gauss equation, they determine the induced Levi--Civita connection and curvature. Accordingly, these two objects provide natural statistical targets for geometric inference.
In this paper, we study their pointwise recovery: given an observation $y$ near $M$, with projection $\pi(y)$, our goal is to construct a linear subspace and a symmetric bilinear map that approximate, respectively, the tangent space $T_{\pi(y)}M$ and the second fundamental form $\Pi_{\pi(y)}$.

A representative approach to tangent space estimation is local principal component analysis (LPCA). In the noiseless setting, \citet{Singer2011VectorDM} analyzes the LPCA covariance matrix and derives asymptotic accuracy guarantees for tangent-space estimation. In contrast, \citet{Tyagi2012} provides sufficient conditions on the neighborhood width and local sample size to guarantee a small principal angle between the PCA subspace and the true tangent space. 
Under additive noise, \citet{Kaslovsky2014NonAsymptoticAO} establishes nonasymptotic perturbation bounds for LPCA in an ambient Gaussian noise model, highlighting a trade-off between noise and curvature. Both \citet{Tyagi2012} and \citet{Kaslovsky2014NonAsymptoticAO} adopt a coordinate-based local model in which tangent coordinates are sampled uniformly and then mapped to the manifold via a local graph parametrization.
Beyond linear fitting, \citet{Cheng2016} shows that incorporating a quadratic term in the local approximation can improve the angular error from first-order to second-order, while \citet{Aamari2017NonasymptoticRF} studies joint estimation of the tangent space and higher-order local polynomial terms. In parallel with LPCA, several recent approaches exploit Laplacian structure to infer tangent directions.
\citet{Kohli2025RobustTS} estimate tangent spaces by orthogonalizing gradients of low-frequency graph-Laplacian eigenvectors. \citet{Jones2025ManifoldDG} first estimates the manifold Laplacian and then uses the carré du champ to recover the induced metric and tangent spaces.

For estimating the second fundamental form, a natural starting point is that under a suitable local parametrization, the quadratic term encodes the second fundamental form.
Along this line, \citet{Aamari2017NonasymptoticRF} fits local polynomials to a local graph parametrization, where the quadratic term provides an approximation of the second fundamental form. They further establish minimax lower bounds for estimating both the tangent space and the second fundamental form under a bounded tubular-noise model. In the noiseless case, the accuracy is governed by the Taylor remainder of the local parametrization; with additional noise, noise deviations enter the local polynomial fit as a controlled perturbation, whose effect is amplified only through the local scale.
From a dual viewpoint, since the shape operator (Weingarten map) is equivalent to the second fundamental form via the ambient metric, one may estimate curvature through normal variations. \citet{Cao2019EfficientWM} proposes a two-step estimator that first estimates tangent and normal spaces by LPCA and then fits the Weingarten map by least squares in the estimated local frame. Their asymptotic rate analysis is carried out under the assumption of exact tangent spaces and a given normal vector field.
Beyond local regression, \citet{Buet2019WeakAA} adopts a varifold perspective and introduces a weak second fundamental form defined via suitable variation operators, together with regularized approximate curvature tensors that are computable for general varifolds, including point cloud varifolds; convergence guarantees are then expressed in terms of the discrepancy between the input varifold and the target submanifold and the regularization scale.
\citet{Jones2025ManifoldDG} exploits the diffusion-geometric framework for second-order geometry, combining tangent and normal information with Hessian-level carré du champ identities to recover the second fundamental form.

In contrast to approaches based on local coordinates, bounded tubular noise, or manifold estimation, we study Gaussian-corrupted observations generated by convolving the uniform measure on a smooth compact submanifold $M\subset\mathbb{R}^D$ with ambient Gaussian noise. Our analysis shows that fundamental Riemannian quantities of $M$, including tangent spaces, the intrinsic dimension, and the second fundamental form, are identifiable from derivatives of the noisy density in a tubular neighborhood of $M$. This is achieved by deriving uniform small-noise expansions for the density and its derivatives, which then lead to population- and sample-level estimators of these geometric objects. We also develop a geometric interpretation of these constructions through density-induced ambient metrics.

Our approach is built on the small-noise structure of the noisy density. In this regime, the log-density behaves like a scaled squared distance to $M$, and its first three derivatives encode, respectively, a normal field, a spectral splitting between tangent and normal directions, and the variation of the tangential projection. This leads to plug-in geometric estimators based on derivatives of the log-density up to order three. At the sample level, we estimate these derivatives by kernel methods in $\mathbb{R}^D$ and derive convergence rates for the resulting geometric estimators.

The paper is organized as follows. Section~\ref{sec:prelim} introduces the geometric setup and preliminaries. Section~\ref{sec:density-and-derivatives} develops uniform small-noise expansions for the density and its derivatives in a tubular neighborhood. Section~\ref{sec:compute-geometry} uses these expansions to construct population-level estimators of tangent spaces, the intrinsic dimension, and the second fundamental form. Section~\ref{sec:metric-change} presents the density-induced metric interpretation. Section~\ref{sec:sample-level-estimator} establishes sample-level guarantees via kernel estimation of density derivatives. Section~\ref{sec:numerical} reports numerical studies.

\begin{figure}[!t]
    \centering
    \includegraphics[width=0.9\linewidth]{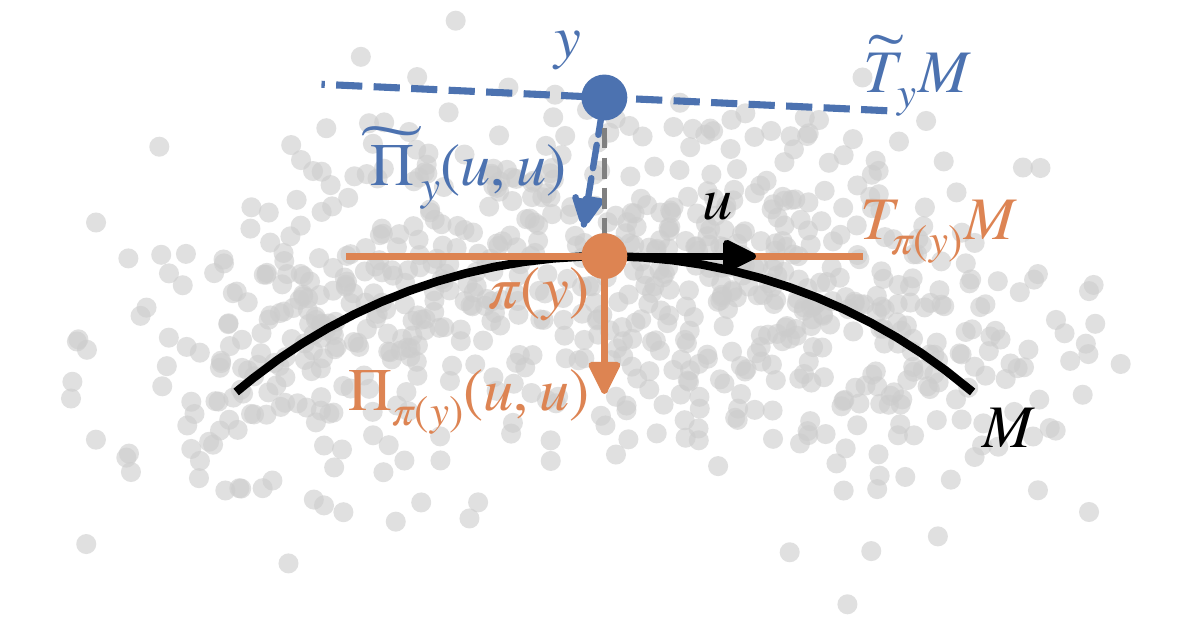}
    \caption{Illustration of Riemannian geometry estimation from noisy data. The underlying submanifold $M$ (black) is observed through noisy samples (gray). For a point $y$ near $M$, let $\pi(y)$ denote the projection of $y$ onto $M$. For a tangent vector $u$, the true tangent space and second fundamental form at $\pi(y)$ are $T_{\pi(y)}M$ and $\Pi_{\pi(y)}(u,u)$, respectively, while the estimated tangent space and second fundamental form at $y$ are denoted by $\widetilde{T}_yM$ and $\widetilde{\Pi}_y(u,u)$.}
    \label{fig:visual-main}
\end{figure}

\subsection{Main Results}
We now summarize the main contributions of the paper. Throughout, $M\subset\mathbb{R}^D$ is a compact $C^k$ ($k\ge 6$), $d$-dimensional submanifold with positive reach $\tau>0$, and the clean data $\mathcal X$ are sampled uniformly from $M$. We observe
\[
\mathcal Y=\mathcal X+\xi,\qquad \mathcal X\sim P_M,\quad \xi\sim \mathcal N(0,\sigma^2 I_D),
\]
and write $P_\sigma=P_M* \mathcal N(0,\sigma^2 I_D)$ for the resulting density. For $y$ in a fixed tubular neighborhood $T(\tau-\varepsilon)$, let $\pi(y)$ be its nearest-point projection onto $M$ and let $v_y=y-\pi(y)$ be the normal displacement.

Our main observation is that tangent spaces, the intrinsic dimension, and the second fundamental form of $M$ are identifiable from the derivatives of the log-density
\[
\mathcal G_\sigma(y) = \nabla\log P_\sigma(y),\qquad
\mathcal H_\sigma(y) = \nabla^2\log P_\sigma(y),\qquad
\mathcal T_\sigma(y) = \nabla^3\log P_\sigma(y).
\]
Based on this, we establish both population-level identification results and sample-level statistical rates. We use $\widehat{(\cdot)}$ for population-level constructions and $\widetilde{(\cdot)}$ for their sample-level counterparts.
Figure~\ref{fig:visual-main} provides a schematic illustration of the proposed geometric estimators.

At the population level, we first derive uniform small-noise expansions for $P_\sigma$ and the derivatives of $\log P_\sigma$ on $T(\tau-\varepsilon)$. 
These expansions show that the Hessian $\mathcal H_\sigma(y)$ has $d$ eigenvalues corresponding to tangent directions and $D-d$ eigenvalues corresponding to normal directions, separated by a gap of order $\sigma^{-2}$. As a consequence, the tangent space can be recovered from the eigenspace associated with the largest $d$ eigenvalues of $\mathcal H_\sigma(y)$, for $y\in T(\tau-\varepsilon)$ with error
\[
\sin\{\Theta(\widehat T_yM,T_{\pi(y)}M)\}=O(\sigma^2).
\]
The same spectral splitting also yields a consistent estimator of the intrinsic dimension $d$.

Next, we construct estimators of the second fundamental form from density derivatives, with different formulas depending on the geometric setting.
\begin{itemize}
  \item For totally umbilical submanifolds, the gradient $\mathcal G_\sigma$ restricted to $M$ recovers the mean curvature vector and thereby yields an $O(\sigma)$ estimator of the second fundamental form.
  \item For hypersurfaces, combining the gradient $\mathcal G_\sigma(y)$ and Hessian $\mathcal H_\sigma(y)$ yields an estimator satisfying
  \[
    \|\widehat\Pi_y-\Pi_{\pi(y)}\|_{\mathrm{op}}
    =O(\|v_y\|+\sigma),
  \]
  over subsets $U_\sigma\subset T(\tau-\varepsilon)$ on which $\|\mathcal G_\sigma(y)\|$ is bounded below.
  \item In arbitrary codimension, differentiating the Hessian-based tangential projector gives an estimator on $T(\tau-\varepsilon)$ based on $(\mathcal H_\sigma,\mathcal T_\sigma)$, with error
  \[
    \|\widehat\Pi_y-\Pi_{\pi(y)}\|_{\mathrm{op}}
    =O(\|v_y\|+\sigma^2),
    \qquad y\in T(\tau-\varepsilon).
  \]
\end{itemize}
Thus, tangent spaces, the intrinsic dimension, and the second fundamental form are all identifiable from derivatives of the noisy log-density.

At the sample level, given i.i.d.\ samples $y_1,\dots,y_N\sim P_\sigma$, we estimate $P_\sigma$ and its derivatives by kernel methods in $\mathbb{R}^D$ and plug them into the population constructions. This yields uniform sample-level estimators for the tangent space and second fundamental form. Moreover, the same eigengap argument also yields consistency of the sample-level intrinsic dimension estimator. Under suitable bandwidth choices, the tangent-space estimator satisfies
\[
\sin\{\Theta(\widetilde T_yM,T_{\pi(y)}M)\}
=
O(\sigma^2)
+
O_P\big(\big(\frac{\log N}{N}\big)^{\frac{2}{D+8}}\big),
\]
while the second fundamental form estimator satisfies
\[
\|\widetilde\Pi_y-\Pi_{\pi(y)}\|_{\mathrm{op}}
=
\begin{cases}
O(\|v_y\|+\sigma)
+
O_P\big(\big(\frac{\log N}{N}\big)^{\frac{2}{D+8}}\big),
& \text{for hypersurfaces},\\[0.75em]
O(\|v_y\|+\sigma^2)
+
O_P\big(\big(\frac{\log N}{N}\big)^{\frac{2}{D+10}}\big),
& \text{in arbitrary codimension}.
\end{cases}
\]
These are ambient-dimensional nonparametric rates, reflecting the fact that our method estimates derivatives of the ambient density directly, rather than first fitting an intrinsic manifold representation.

Finally, we show that the above density-based constructions admit a geometric interpretation through ambient metrics. In particular, suitable density-induced metrics on $\mathbb{R}^D$ make the intrinsic geodesics of $M$ asymptotically compatible with ambient trajectories as the noise level vanishes. For totally umbilical submanifolds and hypersurfaces, we construct, respectively, a Riemannian metric and a degenerate metric:
\[
g_{P_\sigma} = P_\sigma^{4/d} g_E, \qquad g_{P_\sigma} = d(\log P_\sigma)\otimes d(\log P_\sigma).
\]

\section{Notation and Preliminaries}\label{sec:prelim}
In this section, we collect the geometric notation and standard facts that are used throughout the paper. Proofs of the lemmas are provided in Appendix~\ref{app:lemma-sec2}.

Throughout the paper we fix integers $d, D$, and $k\ge6$, and consider a compact $C^k$ $d$-dimensional submanifold $M\subset \mathbb{R}^D$. The ambient space $\mathbb{R}^D$ is equipped with the Euclidean metric $g_E = \langle\cdot, \cdot \rangle$. The induced Riemannian metric on $M$ is denoted by $g:=g_E|_M$. We write $V_M$ for the volume of $M$ under the induced volume measure. 
For a Euclidean space $E$, a point $z\in E$, and $r>0$, we write $B_E(z,r)$ for the open Euclidean ball of radius $r$ centered at $z$. In particular, when $E=\mathbb R^m$ we also write $B_m(z,r)$.

\subsection{Geometric Setup and Reach}
For a closed set $A\subset \mathbb{R}^D$ we recall the notion of reach.
\begin{definition}\label{def:reach}
    Let $A$ be a closed subset of $\mathbb{R}^D$. The reach of $A$ is the largest $\tau\ge 0$ such that every point at distance less than $\tau$ from $A$ has a unique nearest point in $A$.
\end{definition}
Denote the reach of $M$ by $\tau$. We assume throughout that $M$ has positive reach $\tau > 0$. For points $x \in M$ we write $T_x M$ and $T_x^\perp M$ for the tangent and normal spaces at $x$, respectively. The tangent and normal bundles are denoted by $TM$ and $T^\perp M$. 

For $\tau > 0$ we define the open $\tau$-tubular neighborhood of $M$ in the normal bundle by
\[
  \mathcal{T}(\tau) := \{ (x,v) \in T^\perp M : \|v\| < \tau \}.
\]
We identify $(x,v) \in \mathcal{T}(\tau)$ with the point $x + v \in \mathbb{R}^D$. For $0<\varepsilon<\tau$ we also introduce the closed tubular neighborhood
\[
  T(\tau - \varepsilon) := \{ (x,v) \in T^\perp M : \|v\| \le \tau - \varepsilon \}.
\]
Since $M$ is compact and $\tau > 0$, the set $T(\tau - \varepsilon)$ is compact. We use $\mathcal T(\tau)$ for the open tubular neighborhood where the nearest-point projection is defined, and $T(\tau-\varepsilon)$ for its closed inner truncation on which all uniform estimates below are stated.

The exponential map at $x \in M$ is denoted by $  \exp_x\colon  T_x M \to M$. The injectivity radius at $x \in M$, denoted by $\mathrm{inj}(x)$, is the largest radius $r>0$ such that $\exp_x\colon B_{T_xM}(0,r)\to M$ is a diffeomorphism, and the global injectivity radius of $M$ is $\mathrm{inj}_M := \inf_{x \in M} \mathrm{inj}(x).$

The reach provides a lower bound on the injectivity radius.
\begin{lemma}\label{lem:injectivity-radius} \citep[Corollary~4]{Alexander2005GaussEA}
    Let $M$ be a $C^2$ embedded submanifold with reach $\tau>0$. Then
    \[
    \mathrm{inj}_M \ge \pi\tau.
    \]
\end{lemma}

\subsection{Tubular Neighborhoods and Projection}
For a set $A \subset \mathbb{R}^D$ and a point $z \in \mathbb{R}^D$, we write $  d(z,A) := \min_{a \in A} \|z - a\|$ for the Euclidean distance from $z$ to $A$. When the nearest point is unique we denote the corresponding projection by $\pi_A(z) := \arg \min_{a \in A} \|z - a\|$. By default, for each $y\in \mathcal T(\tau)$, we write $\pi(y)$ for the nearest-point projection onto $M$ and let $v_y := y - \pi(y)$ denote the corresponding normal displacement.

The positive reach assumption ensures that the projection $\pi \colon \mathcal T(\tau) \to M$ is well-defined. Since $M$ is $C^k$ with $k \ge 6$, the tubular neighborhood theorem implies that $\pi$ and $v_y$ are $C^{k-1}$ on $\mathcal T(\tau)$. 
Moreover, we write
\[
P_T(y):\mathbb R^D\to T_{\pi(y)}M,
\qquad
P_N(y):\mathbb R^D\to T_{\pi(y)}^\perp M
\]
for the orthogonal projections onto the tangent and normal spaces at $\pi(y)$, respectively. Since $P_T(x)$ and $P_N(x)$ are $C^{k-1}$ in $x\in M$ and $\pi$ is $C^{k-1}$ on $T(\tau-\varepsilon)$, $P_T(y)$ and $P_N(y)$ are $C^{k-1}$ on $T(\tau-\varepsilon)$. 

\subsection{Exponential Coordinates and Curvature}
For $r < \mathrm{inj}(x)$, the exponential map introduces an exponential coordinate chart.
In particular, the map $(x,u)\mapsto \exp_x(u)$, its local inverse $(x,z)\mapsto \exp_x^{-1}(z)$, and the metric coefficients in exponential coordinates $(x,u)\mapsto g_{ij}(x,u)$ are $C^{k-2}$.

We record the following standard expansions under the standing $C^k$ assumption. 
\begin{lemma}\label{lem:exponential-map-expansion}
There exist constants $r_0\in(0,\mathrm{inj}_M)$ and $C$ such that, for every $x\in M$ and every $u\in T_xM$ with $\|u\|\le r_0$,
\[
\exp_x(u)
=
x+u+\frac12\Pi_x(u,u)+Q_e(x,u),
\qquad
\|Q_e(x,u)\|\le C\|u\|^3.
\]
\end{lemma}

\begin{lemma}\label{lem:volume-element-expansion}
There exist constants $r_0\in(0,\mathrm{inj}_M)$ and $C$ such that, for every $x\in M$ and every $u\in T_xM$ with $\|u\|\le r_0$,
\[
\sqrt{\det g_{ij}(x,u)}
=
1+Q_g(x,u),
\qquad
|Q_g(x,u)|\le C\|u\|^2.
\]
\end{lemma}
We denote by $\nabla$ the ambient derivative in $\mathbb{R}^D$ (i.e.,\ the Levi--Civita connection of $g_E$), and by $\nabla^M$ the Levi--Civita connection on $(M,g)$. The second fundamental form at $x \in M$ is the symmetric bilinear map
\[
  \Pi_x \colon T_x M \times T_x M \to T_x^\perp M, \qquad
  \Pi_x(X,Y) = (\nabla_X Y)^\perp,
\]
where $(\cdot)^\perp$ denotes orthogonal projection onto $T_x^\perp M$. In particular, $x\mapsto \Pi_x$ has $C^{k-2}$ regularity. For a normal vector $n \in T_x^\perp M$, the associated shape operator $S_n \colon T_x M \to T_x M$ is given by $S_n(X) = -(\nabla_X n)^\top$, where $(\cdot)^\top$ denotes the projection onto $T_x M$. The mean curvature vector at $x$ is the trace $H_x = \frac{1}{d}\operatorname{Tr}(\Pi_x) \in T_x^\perp M$. We denote the operator norm by $\|\cdot\|_{\mathrm{op}}$. For multilinear maps, $\|\cdot\|_{\mathrm{op}}$ is the usual supremum over unit inputs.

Positive reach implies a uniform bound on the second fundamental form:
\begin{lemma}\label{lem:second-fundamental-form-bound} \citep[Proposition~6.1]{Niyogi2008FindingTH}
Let $M$ be an embedded submanifold with reach $\tau > 0$. Then
\[
  \sup_{x\in M}\|\Pi_x\|_{\mathrm{op}} \le \tau^{-1}.
\]
\end{lemma}

\subsection{Generic Notation}
For two linear subspaces $E,F\subset \mathbb R^D$, we write $\Theta(E,F)$ for their largest principal angle, so that $\sin\{\Theta(E,F)\}$ denotes the corresponding subspace discrepancy.

Unless stated otherwise, generic constants $C, C_1, C_2, \dots$ and $O(\cdot)$ bounds in Sections~\ref{sec:density-and-derivatives}--\ref{sec:metric-change} are uniform over $y\in T(\tau-\varepsilon)$ and $0<\sigma\le \sigma_0$ for sufficiently small $\sigma_0$, and may depend on $M$, $d$, $D$, $k$, $\tau$, and $\varepsilon$. In Section~\ref{sec:sample-level-estimator}, where $\sigma$ is fixed, uniformity is only with respect to $y\in T(\tau-\varepsilon)$, and the implied constants may depend on $\sigma$.

\section{Asymptotic Expansions of the Noisy Density and Its Derivatives}\label{sec:density-and-derivatives}
In this section, we study the small-noise structure of the manifold-generated density $P_\sigma = P_M* \mathcal N(0,\sigma^2 I_D)$ on the tubular neighborhood $T(\tau-\varepsilon)$.
Our goal is to obtain uniform expansions of $P_\sigma$ and of the derivatives of $\log P_\sigma$ that we later use to recover tangent spaces and the second
fundamental forms. Throughout, $M$ is a compact $C^k$ submanifold with $k\ge 6$ and positive reach $\tau>0$, and $0<\varepsilon<\tau$ is fixed. For each fixed $\varepsilon$, we work in the small-noise regime $0<\sigma\le\sigma_0$, where $\sigma_0>0$ is sufficiently small.

Proofs of the theorems and corollaries are given in Section~\ref{sec:pf-density-derivative}; the proof of Lemma~\ref{lem:derivative-of-v-pi} is given in Appendix~\ref{app:lemma-sec3}.

\subsection{Manifold-Generated Noisy Density}
Let $P_M$ denote the uniform probability measure on $M$, namely $dP_M(x)=V_M^{-1}d\mu(x)$, where $\mu$ is the induced volume measure on $M$. Let $\xi\sim\mathcal N(0,\sigma^2I_D)$ and let $\Phi_\sigma$ be its density. Then the noisy observation $\mathcal Y= \mathcal X+\xi$ has density
\[
P_{\sigma}(y)= \int_M \Phi_\sigma(y-x)dP_M(x) = \frac{1}{V_M(2\pi\sigma^2)^{\frac{D}{2}}} \int_M \exp\big(-\frac{\|y-x\|^2}{2\sigma^2}\big)d\mu(x). 
\]
Recall that $\pi(y)$ denotes the projection onto $M$ and $v_y$ the corresponding normal displacement. For $\Pi_{\pi(y)}$ the second fundamental form at $\pi(y)\in M$, we define a self-adjoint map $A_y \colon T_{\pi(y)}M \to T_{\pi(y)}M$ by
\[
A_y = I_{T_{\pi(y)}M} - \langle v_y, \Pi_{\pi(y)}\rangle,
\quad
\langle A_y u, v\rangle
=
\langle u,v\rangle - \langle v_y,\Pi_{\pi(y)}(u,v)\rangle.
\]
Since $y \mapsto \pi(y)$ is $C^{k-1}$ on $T(\tau-\varepsilon)$ and $x\mapsto \Pi_x$ is $C^{k-2}$ on $M$, for $k\ge 6$ the map $y\mapsto A_y$ is at least $C^3$.

The following theorem gives an expansion of $P_\sigma(y)$.
\begin{theorem}\label{th:probability-noisy}
Uniformly for $y\in T(\tau-\varepsilon)$ and $0<\sigma\le\sigma_0$,
\[
P_\sigma(y)
=\frac{1}{V_M(2\pi\sigma^2)^{\frac{D-d}{2}}}
\exp\big(-\frac{\|v_y\|^2}{2\sigma^2}\big)
\frac{1}{\sqrt{\det A_y}}
\big(1+O(\sigma\|v_y\| + \sigma^2)\big).
\] 
\end{theorem}
The proof is based on Taylor expansions in exponential coordinates around $\pi(y)$. The leading factor separates the Gaussian decay in the normal direction from the curvature correction $(\det A_y)^{-1/2}$.

For the analysis of derivatives, it is convenient to rewrite Theorem~\ref{th:probability-noisy} at the level of the log-density and to record uniform bounds on the remainder and its derivatives.

\begin{corollary}\label{co:log-density}
Uniformly for $y\in T(\tau-\varepsilon)$ and $0<\sigma\le\sigma_0$,
\begin{equation}\label{eq:log density}
\log P_\sigma(y)
=
\log \big(V_M(2\pi\sigma^2)^{\frac{D-d}{2}}\big)^{-1}
-
\frac{\|v_y\|^2}{2\sigma^2}
-
\frac{1}{2}\log\det A_y
+
R(y,\sigma),
\end{equation}
where $R(y,\sigma)$ is $C^3$ on $T(\tau-\varepsilon)$ and its ambient derivatives satisfy
\[
|R(y,\sigma)| \le C_{0}\big(\sigma\|v_y\|+\sigma^2\big),
\quad
\|\nabla_y^m R(y,\sigma)\|_{\mathrm{op}}
\le
C_{m}\sigma,
\qquad
m\in\{1,2,3\}.
\]
\end{corollary}

\subsection{Gradient Expansion}
We now derive an expansion for the gradient
\[
\mathcal G_\sigma(y) = \nabla \log P_\sigma(y), \qquad y \in T(\tau-\varepsilon).
\]
To differentiate the expansion in Corollary~\ref{co:log-density} with respect to the ambient coordinate $y$, we first record the derivatives of $\pi(y)$ and $v_y$. Recall that $\nabla_\omega(\cdot)$ denotes the Euclidean directional derivative in the direction $\omega \in \mathbb{R}^D$, and $\omega^\top$ denotes the orthogonal projection of $\omega$ onto $T_{\pi(y)}M$.
\begin{lemma}\label{lem:derivative-of-v-pi}
    For every $y\in \mathcal T(\tau)$ and every $\omega\in \mathbb{R}^D$,
    \[
    \nabla_\omega\pi(y) = A_y^{-1} \omega^\top, \qquad \nabla_\omega v_y =\omega - A_y^{-1}\omega^\top.
    \]
\end{lemma}
The proof is based on the characterization of $A_y$ in terms of the second fundamental form. Combining Lemma~\ref{lem:derivative-of-v-pi} with the log-density expansion in Corollary~\ref{co:log-density}, we obtain the following expression for the gradient.
\begin{theorem}\label{th:gradient-noisy}
Uniformly for $y\in T(\tau-\varepsilon)$ and $0<\sigma\le\sigma_0$,
    \begin{align}\label{eq:gradient}
        \mathcal G_\sigma(y) = -\frac{v_y}{\sigma^2}+\frac{1}{2}dH_{\pi(y)}+O(\|v_y\|+\sigma),
    \end{align}
    where the $O(\cdot)$ term is taken in the Euclidean norm.
\end{theorem}
The leading term $-v_y/\sigma^2$ is purely normal and comes from the Gaussian convolution, while the first curvature correction is given by the mean curvature vector. In particular, for $x\in M$,
\[
\mathcal G_\sigma(x)=\frac d2 H_x+O(\sigma).
\]
Away from $M$, the gradient is asymptotically aligned with the normal direction.

\subsection{Hessian Expansion}
We next study the Hessian of the log-density:
\[
\mathcal{H}_\sigma(y) = \nabla^2\log P _\sigma(y) = \nabla \mathcal G_\sigma(y).
\]
Using the gradient expansion in Theorem~\ref{th:gradient-noisy} together with Lemma~\ref{lem:derivative-of-v-pi}, we obtain the following description of $\mathcal{H}_\sigma(y)$ in $T(\tau-\varepsilon)$.

\begin{theorem}\label{th:noisy-hessian}
Uniformly for $y\in T(\tau-\varepsilon)$ and $0<\sigma\le\sigma_0$,
\[
\mathcal{H}_\sigma(y)
= - \frac{1}{\sigma^2}P_N(y)
- \frac{1}{\sigma^2}(I_{T_{\pi(y)}M} - A_y^{-1})P_T(y)
+ O(1),
\]
where the $O(1)$ term is taken in the operator norm.
\end{theorem}
Here $P_T(y)$ and $P_N(y)$ are orthogonal projections as in Section~\ref{sec:prelim}.
Define the leading-order operator $\mathcal{H}_0(y) = - \frac{1}{\sigma^2}P_N(y) - \frac{1}{\sigma^2}(I_{T_{\pi(y)}M} - A_y^{-1})P_T(y)$. We next record a corollary on the spectral gap of $\mathcal H_0$ and $\mathcal H_\sigma$.
\begin{corollary}\label{co:noisy-hessian-gap}
For each fixed $\varepsilon\in(0,\tau)$, there exist constants $c_\varepsilon>0$ and $\sigma_0>0$ such that the following holds. Let $\lambda_1(y,\sigma)\ge \cdots \ge \lambda_D(y,\sigma)$
be the eigenvalues of $\mathcal H_\sigma(y)$. Then, for every $y\in T(\tau-\varepsilon)$ and every $0<\sigma\le \sigma_0$,
\[
\lambda_d(y,\sigma)-\lambda_{d+1}(y,\sigma)\ge c_\varepsilon \sigma^{-2}.
\]
The same spectral gap holds for $\mathcal H_0(y)$.
\end{corollary}

\section{Population Geometric Estimators from Density Derivatives}\label{sec:compute-geometry}
In this section, we develop population-level identification formulas and estimators for geometric quantities of the embedded manifold $M$ from derivatives of the noisy log-density $\log P_\sigma$. At this stage we assume access to the gradient $\mathcal G_\sigma = \nabla \log P_\sigma$, the Hessian $\mathcal H_\sigma = \nabla^2 \log P_\sigma$, and the third-order derivative $\mathcal T_\sigma = \nabla^3 \log P_\sigma$. Fix $\varepsilon \in (0,\tau)$ and let $\sigma_0>0$ be as in Section~\ref{sec:density-and-derivatives}; throughout this section we consider $y\in T(\tau-\varepsilon)$ and $0<\sigma\le \sigma_0$, so that the expansions of Theorems~\ref{th:probability-noisy}, \ref{th:gradient-noisy}, \ref{th:noisy-hessian}, Corollary~\ref{co:log-density}, and Corollary~\ref{co:noisy-hessian-gap} apply.

Our results reveal a simple hierarchy in the geometric information encoded by density derivatives. Second-order derivatives identify tangent spaces and the intrinsic dimension by a spectral gap. For totally umbilical submanifolds, the second fundamental form collapses to the mean curvature vector, so first-order derivatives already suffice. For hypersurfaces, first- and second-order derivatives recover the second fundamental form through an estimated normal direction. In arbitrary codimension, recovering the full second fundamental form requires differentiating the tangent projector and therefore involves third-order derivatives. We present the constructions in this order of increasing geometric complexity.

The second fundamental form estimators in this section are population-level constructions. They are defined as bilinear forms on the true tangent space $T_{\pi(y)}M$, and are analyzed relative to the corresponding population target. The sample-level counterparts are obtained in Section~\ref{sec:sample-level-estimator} by replacing $\mathcal G_\sigma$, $\mathcal H_\sigma$, and $\mathcal T_\sigma$ with plug-in estimators based on noisy observations, together with estimated tangent spaces when needed.

Proofs of theorems and corollaries are given in Section~\ref{sec:pf-geometry}; proofs of lemmas are given in Appendix~\ref{app:lemma-sec4}.

\subsection{Tangent Spaces and the Intrinsic Dimension Estimation}
By Theorem~\ref{th:noisy-hessian} and Corollary~\ref{co:noisy-hessian-gap}, the $d$ tangential eigenvalues of the Hessian $\mathcal H_\sigma(y)$ are separated from the remaining $D-d$ normal eigenvalues by a spectral gap of order $\sigma^{-2}$.
It is therefore natural to estimate the tangent space at $\pi(y)$ by the eigenvectors associated with the $d$ largest eigenvalues, and to estimate the intrinsic dimension $d$ by the location of the largest spectral gap.

Formally, let $\lambda_1 \geq \lambda_2 \geq \dots \geq \lambda_D$ be the eigenvalues of $\mathcal{H}_\sigma(y)$ and let $e_1,\dots,e_D$ be the associated orthonormal eigenvectors. If the intrinsic dimension $d$ is known, we define the tangent space estimator by
\[
\widehat{T}_{y}M = \mathrm{span}\{e_1,\dots,e_d\}.
\]
If $d$ is unknown, we estimate it by the location of the largest consecutive spectral gap,
\[
\widehat{d} \in \arg\max_{k=1,\dots,D-1} |\lambda_k - \lambda_{k+1}|.
\]
The following theorem shows that the tangent space estimator is consistent when $d$ is known, and that the spectral-gap estimator recovers the correct intrinsic dimension in the small-noise regime.
\begin{theorem}\label{th:tangent-space}
Let $\widehat T_yM$ be the eigenspace associated with the $d$ largest eigenvalues of $\mathcal H_\sigma(y)$. Then there exist constants $C$ and $\sigma_0>0$ such that, for every $y\in T(\tau-\varepsilon)$ and every $0<\sigma\le\sigma_0$,
\[
\sin\{\Theta(\widehat{T}_{y}M,T_{\pi(y)}M)\} \le C\sigma^2.
\]
Moreover, for any sequence with $\|v_y\|\to0$ and $\sigma\to0$, the largest consecutive spectral gap of $\mathcal H_\sigma(y)$ is attained at $k=d$. Consequently, $\widehat d=d$ along the sequence.
\end{theorem}
We note that \citet{Stanczuk2024} exploits the small-noise geometry of the gradient field $\nabla \log P_\sigma$ to estimate tangent spaces and the intrinsic dimension. Relying on the property that $\nabla \log P_\sigma$ is asymptotically normal, they propose to sample enough gradient vector in a small neighborhood of a point so as to span the normal space. Then the spectrum of this stacked gradient matrix tells tangent spaces and the intrinsic dimension.
In contrast, our estimators are based on the local spectral structure of the Hessian, which enables recovery from pointwise derivative information.

\subsection{Second Fundamental Form for Totally Umbilical Submanifolds}
We now turn to the estimation of the second fundamental form. We begin with the class of totally umbilical submanifolds, characterized by
\[
\Pi(u,v) = g(u,v)H, 
\]
for all tangent vectors $u,v$, where $H$ is the mean curvature and $g$ is the induced metric. In Euclidean space, this class consists of affine subspaces and spheres of arbitrary dimension. Following Theorem~\ref{th:gradient-noisy}, for $x\in M$ we have
\[
  \mathcal G_\sigma(x)
  = \frac{d}{2} H_x + O(\sigma).
\]
Therefore at noiseless points on the manifold the gradient recovers the mean curvature vector up to the known factor $d/2$ and an $O(\sigma)$ error. Combining this with the definition of totally umbilical submanifolds suggests the following estimator.
\begin{theorem}\label{th:umbilical-submanifold}
Assume $M$ is totally umbilical. For $x\in M,\ u,v \in T_xM$, define
\[
\widehat{\Pi}_x(u,v) = \frac{2}{d}g(u,v)\mathcal G_\sigma(x).
\]
Then there exist constants $C$ and $\sigma_0>0$ such that, for every $x\in M$ and $0<\sigma\le\sigma_0$,
\[
\|\widehat{\Pi}_x-\Pi_x\|_{\mathrm{op}} \le C \sigma.
\]
\end{theorem}
Due to the presence of the leading normal term $-v_y/\sigma^2$ in the gradient expansion (Theorem~\ref{th:gradient-noisy}), this estimator cannot be directly generalized from $x\in M$ to noisy points $y\in T(\tau-\varepsilon)$ uniformly in $\sigma$. Unless $\|v_y\|$ is of order $o(\sigma^2)$, the Gaussian term $-v_y/\sigma^2$ dominates the curvature contribution and prevents a stable recovery of $\Pi_{\pi(y)}$ from the gradient alone.

\subsection{Second Fundamental Form for Hypersurfaces}\label{sec:Pi-hypersurface construction}
We next consider hypersurfaces, i.e.,\ submanifolds of codimension one ($D =d+1$). In this setting, the second fundamental form can be reconstructed from a normal vector field. 
\begin{lemma}\label{lem:extended-normal-vector-field}
Let $M$ be a hypersurface, and let $\mathcal N$ be a $C^1$ vector field defined on $\mathcal T(\tau)$ such that $\mathcal N(y) \in T_{\pi(y)}^\perp M$, $\mathcal N(y)\ne 0$ for all $y\in \mathcal T(\tau)$. Then, for every $y\in\mathcal T(\tau)$ and $u,v\in T_{\pi(y)}M$,
\[
\Pi_{\pi(y)}(u,v) = -\frac{1}{\|\mathcal N(y)\|^2}\big\langle \nabla_{A_y u}\mathcal N(y), v \big\rangle \mathcal N(y).
\]
\end{lemma}
Here $A_y$ is the tangent endomorphism introduced in Section~\ref{sec:density-and-derivatives}. In particular, $A_yu = u - (\nabla_uv_y)^\top$.
According to Theorem~\ref{th:gradient-noisy}, the tangential component of the gradient
$\mathcal G_\sigma=\nabla\log P_\sigma$ is of lower order than its normal component near
$M$, so $\mathcal G_\sigma$ provides a natural approximate normal field. This suggests
replacing the true normal field in Lemma~\ref{lem:extended-normal-vector-field} by
$\mathcal G_\sigma$. Moreover, using the definition of the Hessian $\mathcal{H}_\sigma(u,v):=\langle\nabla_u \nabla \log P_\sigma, v\rangle$, we give the following theorem.
\begin{theorem}\label{th:hypersurface}
Assume $M\subset\mathbb{R}^D$ is a hypersurface. Let $\{U_\sigma\}_{0<\sigma\le\sigma_0}$ be a family of subsets of $T(\tau-\varepsilon)$ such that
\[
\inf_{y\in U_\sigma}\|\mathcal G_\sigma(y)\|\ge c_0>0
\qquad\text{for every }0<\sigma\le\sigma_0.
\]
For $y\in U_\sigma,\ u,v\in T_{\pi(y)}M$, define
\[
\widehat{\Pi}_y(u,v) = -\frac{1}{\|\mathcal G_\sigma(y)\|^2}\big\langle\mathcal{H}_\sigma(y)u, v\big\rangle \mathcal G_\sigma(y).
\]
Then there exists a constant $C$ such that, for every $y\in U_\sigma$ and $0<\sigma\le\sigma_0$,   
\[
\|\widehat{\Pi}_y-\Pi_{\pi(y)}\|_{\mathrm{op}} \le C(\|v_y\|+\sigma).
\]
\end{theorem}
On regions where $\mathcal G_\sigma$ may degenerate, one may instead use the projector-based estimator of Theorem~\ref{th:general-manifold}, which also applies to hypersurfaces.

The same formula, with $M$ replaced by the level set $L_{\sigma,c}:=\{y\in\mathbb R^D: P_\sigma(y)=c\}$, yields the second fundamental form of $L_{\sigma,c}$ at $y$ without error whenever $\mathcal G_\sigma(y)\neq 0$, since $\mathcal G_\sigma(y)$ is strictly normal to the level set and Lemma~\ref{lem:extended-normal-vector-field} applies with $\mathcal N = \mathcal G_\sigma$. With regard to the density level set, we can measure how close it is to $M$ following the expansion of $P_\sigma$.
\begin{corollary}\label{co:level-set}
    There exist constants $C$ and $\sigma_0>0$ such that, for every level set $L_{\sigma,c}$ with $L_{\sigma,c}\cap M\neq\emptyset$ and $0<\sigma\le\sigma_0$,
    \[
    \sup_{y\in L_{\sigma,c}} d(y,M)=O(\sigma^2).
    \]
\end{corollary}
This suggests that noisy density generates nearby geometric surrogates of $M$ through its level sets, which will be further discussed in Section~\ref{sec:metric-change} with regard to the density-induced metrics.

\subsection{Second Fundamental Form for Submanifolds of Arbitrary Codimension}\label{sec:Pi-codimension}
For submanifolds of arbitrary codimension in $\mathbb{R}^D$, there is no unique normal direction, and the second fundamental form must be recovered in a different way. We refer to this as the general-codimension case. Due to the embedding relation, the derivatives on $M$ can be viewed as the tangential projection of the derivatives on $\mathbb{R}^D$. Hence, we may expect that the projection operator embodies the curvatures of $M$. In the following lemma we introduce an approach of computing the second fundamental form by taking the directional derivative of the tangential projection map.
\begin{lemma}\label{lem:second-fundamental-form-general}
    Let $P_T(y)$ be a $C^1$ tangential projection field on $\mathcal T(\tau)$. Then, for every $x\in M$ and $u,v\in T_xM$,
    \begin{align}
        \Pi_x(u,v) = \nabla_uP_T(x) v.
    \end{align}
\end{lemma}
By Theorem~\ref{th:tangent-space}, the tangent space at $y$ can be estimated as $\widehat{T}_yM = \mathrm{span}\{e_1,\dots,e_d\}$, where $e_1,\dots,e_d$ are the eigenvectors of $\mathcal{H}_\sigma(y)$ associated with the $d$ largest eigenvalues. We denote the corresponding approximate projection by
\[
\widehat{P}_T(y)= E_yE_y^T,
\]
where $E_y$ is the $D\times d$ matrix with columns $e_1,\dots,e_d$. By Corollary~\ref{co:noisy-hessian-gap}, the $d$ tangential eigenvalues of $\mathcal H_\sigma(y)$ are separated from the $D-d$ normal eigenvalues on $T(\tau-\varepsilon)$ by a gap of order $\sigma^{-2}$. Therefore, since $y \mapsto \mathcal H_\sigma(y)$ is $C^1$, the associated spectral projector $\widehat P_T(y)$ is well-defined and $C^1$ on $T(\tau-\varepsilon)$.

Substituting $\widehat{P}_T(y)$ for $P_T(y)$ in the lemma leads to the following theorem.
\begin{theorem}\label{th:general-manifold}
Let $\widehat P_T(y)$ be the tangential spectral projector associated with the $d$ largest eigenvalues of $\mathcal H_\sigma(y)$. For $y\in T(\tau-\varepsilon),\ u,v\in T_{\pi(y)}M$, define
\begin{equation}\label{eq:pi-general-manifold}
\widehat{\Pi}_y(u,v) = \nabla_{u} \widehat{P}_T(y) v.
\end{equation}
Then there exist constants $C$ and $\sigma_0>0$ such that, for every $y\in T(\tau-\varepsilon)$ and $0<\sigma\le\sigma_0$,
\[
\|\widehat{\Pi}_y-\Pi_{\pi(y)}\|_{\mathrm{op}} \le C(\|v_y\|+\sigma^2).
\]
\end{theorem}
The proof uses perturbation bounds for spectral projectors. Since $\widehat P_T(y)$ is the Riesz projector associated with the tangential eigenspace of $\mathcal H_\sigma(y)$, the difference $\nabla_u\widehat P_T(y)-\nabla_u P_T(y)$ can be controlled in terms of the perturbation of $\mathcal H_\sigma(y)$ and its derivative $\nabla_u\mathcal H_\sigma(y)$.

In practice, for a fixed direction $u$, the derivative
$X:=\nabla_u\widehat P_T(y)$ can be computed from the Sylvester-type system
\[
\mathcal H_\sigma(y)X - X\mathcal H_\sigma(y)
=
\widehat P_T(y)\nabla_u\mathcal H_\sigma(y)
-
\nabla_u\mathcal H_\sigma(y)\widehat P_T(y),\quad \widehat P_T(y)X + X\widehat P_T(y)=X.
\]
Since the tangential and normal spectral clusters are
separated by a gap of order $\sigma^{-2}$, this system has a unique solution.
Moreover, the estimator \eqref{eq:pi-general-manifold} depends on third-order derivatives of the log-density through
\[
\nabla_u\mathcal H_\sigma(y)=\nabla^3 \log P_\sigma(y)(u,\cdot,\cdot).
\]

The approach of computing the second fundamental form via the derivative of the tangential projection goes back to \citet{Hutchinson1986, Ambrosio1998CurvatureAD}. Conceptually, both Theorem~\ref{th:tangent-space} (tangent space) and Theorem~\ref{th:general-manifold} (second fundamental form) exploit the same structural property of the log-density (Corollary~\ref{co:log-density}): in the small-noise regime, the leading term $-{\|v_y\|^2}/{(2\sigma^2)}$ in the expansion of $\log P_\sigma$ is a scaled ``squared distance function'' to the manifold. The squared distance function encodes both the tangent spaces and the second fundamental form, and our method recovers these geometric quantities from the derivatives of $\log P_\sigma$.
 
\section{Density-Induced Metrics}\label{sec:metric-change}
A central conclusion from the previous sections is that the tangent space and second fundamental form of $M$ admit explicit representations in terms of the noisy density $P_\sigma$ and its derivatives. 
This suggests a further geometric question: can these density-based formulas induce an ambient geometric structure on $\mathbb R^D$ that reflects the intrinsic geometry of $M$ together with its embedding?

Specifically, in this section we investigate how to endow the ambient space $\mathbb R^D$ with a density-induced metric $g_{P_\sigma}$ associated with $P_\sigma$ such that the inclusion $(M,g)\hookrightarrow (\mathbb{R}^D, g_{P_\sigma})$ is a totally geodesic map: any geodesic in $M$ with respect to the induced Euclidean metric $g$ remains a geodesic when viewed as a curve in $(\mathbb{R}^D,g_{P_\sigma})$.
From the perspective of geodesic equation, one may seek an ambient structure whose Christoffel symbols, restricted to tangent directions along $M$, cancel the second fundamental form of the Euclidean embedding. 
We record two constructions. In the totally umbilical case, the density induces a conformal Riemannian metric on the ambient space. In the hypersurface case, the resulting metric is degenerate, so the corresponding statement is understood through the associated Christoffel-type coefficient field and its induced acceleration equation.

The results in this section should be viewed as geometric interpretation of the density-based recovery formulas, and are not of the statistical estimation theory developed in the rest of the paper. 
Proofs of Theorem~\ref{th:metric-totally-umbilical-submanifold} and Theorem~\ref{th:metric-hypersurface} are given in Section~\ref{sec:pf-metric}.

\subsection{Metric Change for Totally Umbilical Submanifolds}
Theorem~\ref{th:umbilical-submanifold} shows that the second fundamental form of totally umbilical submanifolds can be recovered from the gradient of log-density $\nabla\log P_\sigma$ up to $O(\sigma)$ error. 
This yields a conformal ambient metric for which the second fundamental form of the inclusion is $O(\sigma)$.
\begin{theorem}\label{th:metric-totally-umbilical-submanifold}
Assume $M$ is totally umbilical. Define the Riemannian metric
\[
g_{P_\sigma} = P_\sigma^{\frac{4}{d}}g_E.
\]
Let $\Pi^{P_\sigma}$ be the second fundamental form of the inclusion $(M,g)\hookrightarrow (\mathbb R^D, g_{P_\sigma})$. Then there exist constants $C$ and $\sigma_0>0$ such that, for every $x\in M$ and $0<\sigma\le\sigma_0$,
\[
\|\Pi^{P_\sigma}_x\|_{\mathrm{op}} \le C\sigma.
\]
In particular, the inclusion is asymptotically totally geodesic as $\sigma\to0$.
\end{theorem}
For the problem of modifying the ambient metric so that geodesics concentrate in regions where the data density is large, a broader line of work considers density-driven conformal metrics of the form $\tilde{g}_{P_\sigma} = w(P_\sigma)g_E$ with factor $w(P_\sigma)>0$ that decreases as $P_\sigma$ increases. In many continuum limits of graph-based distances, such as Fermat-type distances and related constructions, this yields metrics that can be written as $\tilde{g}_{P_\sigma} \propto P_\sigma^{-\alpha}g_E$ for some $\alpha>0$; geodesics under $\tilde{g}_{P_\sigma}$ then tend to avoid low-density regions and to follow high-density zones near the data manifold \citep[e.g.,][]{Groisman2018NonhomogeneousEF,Trillos2024Fermat}. At first sight this seems at odds with our choice $g_{P_\sigma} = P_\sigma^{{4}/{d}}g_E$, whose conformal factor grows with the density. The two constructions, however, encode different geometric quantities. Our conformal factor is determined by the second fundamental form estimator and is chosen precisely so that the intrinsic geodesics on $M$ remain geodesics in the changed ambient space. In contrast, density-inverse metrics such as the Fermat distance arise as continuum limits of shortest-path distances on weighted random geometric graphs. In the noisy manifold setting considered here, both approaches nevertheless bias geodesics toward high-probability regions near $M$ while serving different purposes.

\subsection{Metric Change for Hypersurfaces}
For manifolds of higher curvature complexity, such as hypersurfaces or density level sets, a conformal rescaling is no longer sufficient to incorporate the full second fundamental form. In this case, Theorem~\ref{th:hypersurface} shows that the second fundamental form can be recovered by the gradient of log-density and its derivative. Motivated by this, we consider the degenerate metric
\[
g_{P_\sigma} := d(\log P_\sigma)\otimes d(\log P_\sigma),
\qquad 
g_{P_\sigma}(u,v) = \langle u, \nabla\log P_\sigma\rangle\langle v, \nabla\log P_\sigma\rangle,
\]
where $d(\log P_\sigma)$ is a $1$-form. 
Wherever $\nabla \log P_\sigma(x)\neq 0$, $g_{P_\sigma}$ is a symmetric, positive semi-definite (0,2)-tensor of rank one. 

Since $g_{P_\sigma}$ is degenerate, it does not define the standard Levi–Civita formalism. Nevertheless, the tensor $g_{P_\sigma}$ together with its Moore–Penrose pseudo-inverse, denoted as $(g_{P_\sigma})_{\mathrm{MP}}$, determines a coefficient field of Christoffel type, which in turn yields a second-order system of acceleration. 
\begin{theorem}\label{th:metric-hypersurface}
Assume $M$ is a hypersurface. Define the degenerate metric
\begin{align}\label{eq:metric-gradient}
g_{P_\sigma} = d(\log P_\sigma)\otimes d(\log P_\sigma).
\end{align}
Let $\gamma(t)$ be a geodesic in $(M,g)$, regarded as a curve in $\mathbb{R}^D$, and assume $\nabla\log P_\sigma\neq 0$ on $\gamma(t)$. Define the coefficient field
\[
\widetilde\Gamma^k_{ij}
:=
(g_{P_\sigma})_{\mathrm{MP}}^{mk}
\frac12
\big(
\partial_j(g_{P_\sigma})_{im}
+\partial_i(g_{P_\sigma})_{jm}
-\partial_m(g_{P_\sigma})_{ij}
\big).
\]
Then there exist constants $C$ and $\sigma_0>0$ such that, for every $\gamma$, every $t$ in its domain and $0<\sigma\le\sigma_0$,
\[
\big|\ddot\gamma^k(t)
+
\widetilde\Gamma^k_{ij}(\gamma(t))\dot\gamma^i(t)\dot\gamma^j(t) \big|
\le C\sigma.
\]
\end{theorem}
In this sense, the density-induced degenerate metric $g_{P_\sigma}$ determines a second-order system that matches the geodesic equation of $M$ up to an $O(\sigma)$ error.

Note that for a density level set $L_{\sigma,c}$, the second fundamental form is represented exactly by the same gradient-Hessian formula of hypersurfaces; see Section~\ref{sec:Pi-hypersurface construction}. Hence the above degenerate metric also induces, without approximation error, the corresponding acceleration equation for geodesics on density level sets.

Leveraging the properties of the gradient field $\nabla\log P_\sigma$, such as the fact that it points toward $M$ under small noise and is almost normal to $M$ (Theorem~\ref{th:gradient-noisy}), several recent works construct metrics based on this gradient that encourage geodesics to stay near the data manifold. In particular, \citet{Azeglio2025WhatsIY} propose a metric that is a weighted combination of the Euclidean metric and the degenerate metric $d(\log P_\sigma)\otimes d(\log P_\sigma)$. When the weight on the degenerate metric is large, their metric becomes close to \eqref{eq:metric-gradient}. They empirically observe that boundary-value geodesics under their metric remain close to the data manifold even for high-dimensional data. In the manifold-generated noisy setting considered here, our Theorem~\ref{th:metric-hypersurface} and Corollary~\ref{co:level-set} provide a heuristic explanation for this behavior: when the boundary points lie on the intersection of $M$ and a density level set, the level set stays within an $O(\sigma^2)$ tubular neighborhood when noise is small; since the geodesics of level sets match the acceleration equation of $g_{P_\sigma}$, the resulting curve connecting such boundary points is confined to the same neighborhood.

For submanifolds of arbitrary codimension, the form of the estimator in Theorem~\ref{th:general-manifold} makes an explicit metric construction substantially more difficult.
The Christoffel symbols must satisfy conditions such as the torsion-free and metric-compatibility, and the resulting metric tensor must remain positive definite near $M$. In this sense, Theorem~\ref{th:general-manifold} should be interpreted as providing constraints that any density-based ambient metric would have to satisfy along $M$ in order to encode the correct second fundamental form. 

\section{Sample-Level Estimators}\label{sec:sample-level-estimator}
In this section we work at a fixed noise level $0<\sigma\le\sigma_0$ and estimate the noisy density $P_\sigma$ and its derivatives from i.i.d.\ samples $y_1,\dots,y_N\sim P_\sigma$. All stochastic orders are taken as $N\to\infty$, with constants allowed to depend on $\sigma$. Since $P_\sigma$ is a smooth density on $\mathbb R^D$ for each fixed $\sigma>0$, the sample-level analysis in this section is based on a standard ambient-space kernel density estimator (KDE). These bounds will then be combined in Section~\ref{sec:final bounds} with the population-level constructions from Section~\ref{sec:compute-geometry} to obtain geometric error bounds for the tangent space and the second fundamental form. Since the subsequent geometric constructions are carried out on $T(\tau-\varepsilon)$, we only require uniform control of the KDE and its derivatives on this tubular neighborhood.

\subsection{Kernel Estimators of Noisy Density and Its Derivatives}
Under the manifold-generated noise model of Section~\ref{sec:density-and-derivatives}, $P_\sigma=P_M*\mathcal N(0,\sigma^2 I_D)$ is a bounded density on $\mathbb R^D$. Therefore we estimate $P_\sigma$ and its derivatives by standard ambient-space kernel methods. 
Consider the $D$-dimensional kernel density estimators
\[
\widetilde P_\sigma(y)=\frac{1}{N h^D}\sum_{i=1}^N K\big(\frac{y_i-y}{h}\big),
\qquad y\in\mathbb R^D,
\]
where $K:\mathbb R^D\to\mathbb R$ is a kernel and $h>0$ is the bandwidth. For order $m \in \{1,2,3\}$, we denote the derivatives of the kernel estimator with bandwidth $h_k$  by
\[
G^m(y) = \nabla^m \widetilde{P}_\sigma(y) =  \frac{1}{Nh_m^{D+m}}\sum_{i=1}^N\nabla^m K\big(\frac{y_i-y}{h_m}\big).
\]
We assume $K$ satisfies the standard smoothness,
moment, and entropy conditions under which uniform multivariate KDE bounds hold; see, e.g., \citet{Gine2002RatesOS}. 
\begin{lemma}\label{lem:kernel-density-estimator}
Fix $\varepsilon\in(0,\tau)$ and $0<\sigma\le \sigma_0$. Under the kernel regularity conditions stated above, for each $m\in\{0,1,2,3\}$,
    \begin{align}
        \sup_{y\in T(\tau-\varepsilon)} \|\nabla^m\widetilde{P}_\sigma(y) - \nabla^m P_\sigma(y)\|_{\mathrm{op}} 
        = O(h_m^2) + O_P\big(\sqrt{\frac{\log N}{N h^{D+2m}}}\big).
    \end{align}
 In particular, if
\[
h_m\asymp \big(\frac{\log N}{N}\big)^{{1}/(D+4+2m)},
\]
then
\[
\sup_{y\in T(\tau-\varepsilon)}
\|\nabla^m \widetilde P_\sigma(y)-\nabla^m P_\sigma(y)\|_{\mathrm{op}}
=
O_P\big(
\big(\frac{\log N}{N}\big)^{\frac{2}{D+4+2m}}
\big).
\]
\end{lemma}
The proof of this lemma follows from standard KDE theory. For each fixed $\sigma>0$, the function $P_\sigma$ is smooth, and its derivatives up to order $m+2$ are bounded on the compact set $T(\tau-\varepsilon )$. The stated bias term is the standard second-order kernel bias, while the variance term is the usual uniform empirical-process bound for multivariate KDEs and their derivatives. Combining these yields the claim; see \citet{Gine2002RatesOS, einmahl2005uniform}. For general multivariate density derivative estimation, including asymptotic expansions and bandwidth selection, see \citet{chacon2011asymptotics}; Also see \citet{Genovese2014Ridge,Chen2015ridge} for their use in geometric inference based on derivatives of a KDE.

\subsection{Plug-in Estimators for Log-density Derivatives}
We now define the sample-level estimators of the gradient, Hessian, and third derivative of the log-density. The gradient estimator of the log-density is
\[
\widetilde{\mathcal G}_\sigma(y) = \nabla \log\widetilde{P}_\sigma(y) =\frac{G^1(y)}{\widetilde{P}_\sigma(y)}.
\]
The Hessian is estimated by 
\[
\widetilde{\mathcal{H}}_\sigma(y) = \nabla^2 \log\widetilde{P}_\sigma(y)= \frac{ G^2(y) }{ \widetilde{P}_\sigma(y) } - \widetilde{\mathcal G}_\sigma(y) \otimes \widetilde{\mathcal G}_\sigma(y),
\]
where $\otimes$ denotes the tensor product, and we view $\widetilde{\mathcal G}_\sigma(y) \otimes \widetilde{\mathcal G}_\sigma(y)$ as a $D\times D$ matrix. For the third-order derivative, we have
\[
\widetilde{\mathcal T}_\sigma(y)= \nabla^3\log \widetilde P_\sigma 
= \frac{G^3(y)}{\widetilde{P}_\sigma(y) } - \widetilde{\mathcal{H}}_\sigma(y)\otimes \widetilde{\mathcal G}_\sigma(y) - \widetilde{\mathcal G}_\sigma(y)^{\otimes3},
\]
where $(A\otimes b)_{ijk} = A_{ij}b_k + A_{ik}b_j + A_{jk}b_i$ for $1\le i,j,k\le D$. Combining Lemma~\ref{lem:kernel-density-estimator} with the positivity of $P_\sigma$ on $T(\tau-\varepsilon)$, we obtain the following rates for the log-density derivatives.
\begin{corollary}\label{co:rates-for-log-density-derivatives}
Assume the conditions of Lemma~\ref{lem:kernel-density-estimator}. Then, for each fixed $0<\sigma\le\sigma_0$,
    \begin{align}
           & \sup_{y\in T(\tau-\varepsilon)}\|\widetilde{\mathcal G}_\sigma(y) - \mathcal G_\sigma(y)\| = O_P\big(\big(\frac{\log N}{N}\big)^{\frac{2}{D+6}}\big),\\
            & \sup_{y\in T(\tau-\varepsilon)}\|\widetilde{\mathcal{H}}_\sigma(y)-\mathcal{H}_\sigma(y)\|_{\mathrm{op}} = O_P\big(\big(\frac{\log N}{N}\big)^{\frac{2}{D+8}}\big),\\
            & \sup_{y\in T(\tau-\varepsilon)}\|\widetilde{\mathcal T}_\sigma(y) - {\mathcal T}_\sigma(y)\|_{\mathrm{op}} = O_P\big(\big(\frac{\log N}{N}\big)^{\frac{2}{D+10}}\big).
    \end{align}
\end{corollary}
\begin{proof}
Fix $0<\sigma\le \sigma_0$. Since $T(\tau-\varepsilon )$ is compact and $P_\sigma$ is continuous and strictly positive on it, for each fixed $\sigma$ there exists $c_\sigma>0$ such that $\inf_{y\in T(\tau-\varepsilon )} P_\sigma(y)\ge c_\sigma$. Hence, with probability tending to one, $\widetilde P_\sigma\ge {c_\sigma}/{2}$ uniformly on $T(\tau-\varepsilon )$ by Lemma~\ref{lem:kernel-density-estimator}. 
Moreover, by the uniform bounds on the derivatives of $P_\sigma$ and the corresponding uniform estimation bounds for $\widetilde P_\sigma$, $(p,\nabla p,\nabla^2 p,\nabla^3 p)$ associated with $P_\sigma$ and $\widetilde P_\sigma$ remain, with probability
tending to one, in a compact subset of the region $\{p\ge c_\sigma/2\}$. Since the map from $(p,\nabla p,\nabla^2 p,\nabla^3 p)$ to the log-density
derivatives is smooth on $\{p> 0\}$, it is Lipschitz on any compact subset of $\{p\ge c_\sigma/2\}$.
It follows that the uniform estimation error for the log-density derivatives is bounded by a constant multiple of the corresponding uniform estimation error for the density derivatives.

Applying Lemma~\ref{lem:kernel-density-estimator} for orders $k=0,1,2,3$ yields the stated rates. The displayed rates are obtained by choosing the bandwidth separately for each order.
\end{proof}

\subsection{Sample-Level Geometric Error Bounds}\label{sec:final bounds}
We now plug the estimators $\widetilde{\mathcal G}_\sigma$, $\widetilde{\mathcal{H}}_\sigma$, $\widetilde{\mathcal{T}}_\sigma$ into the population constructions of Section~\ref{sec:compute-geometry} and quantify the resulting geometric errors. 

For tangent space, we define the sample-level  estimator $\widetilde{T}_yM$ as the span of the eigenvectors of $\widetilde{\mathcal{H}}_\sigma(y)$ associated with its $d$ largest eigenvalues. If the intrinsic dimension $d$ is unknown, it may be estimated from the eigengap of $\widetilde{\mathcal H}_\sigma(y)$, in direct analogue with the population construction in Theorem~\ref{th:tangent-space}. In particular, for sufficiently small $\sigma$, the population eigengap remains uniformly separated on $T(\tau-\varepsilon)$, and the above Hessian bound then also yields consistency of $\widetilde d$, uniformly over $y\in T(\tau-\varepsilon)$.

To compare the true and estimated second fundamental forms as bilinear maps on a common domain, we work with their ambient extensions. For each
$y\in T(\tau-\varepsilon)$, let
\[
\Pi_{\pi(y)}(u,v)
:=
\Pi_{\pi(y)}\big(P_T(y)u,P_T(y)v\big),
\qquad u,v\in\mathbb R^D,
\]
denote the ambient extension of the true second fundamental form. 
The extended sample-level estimator is likewise defined using the estimated tangential projector $\widetilde{P}_T(y)$ obtained from the eigenvectors of $\widetilde{\mathcal{H}}_\sigma(y)$ (see Section~\ref{sec:Pi-codimension}). Accordingly, all operator norms for second fundamental forms below are taken over bilinear maps from $\mathbb R^D\times\mathbb R^D$ to $\mathbb R^D$.

In the hypersurface setting ($D-d=1$), following Theorem~\ref{th:hypersurface} we define
\[
\widetilde{\Pi}_y(u,v)
=
-\frac{1}{\|\widetilde{\mathcal G}_\sigma(y)\|^2}
\big\langle
\widetilde{\mathcal H}_\sigma(y)\widetilde P_T(y)u,
\widetilde P_T(y)v
\big\rangle
\widetilde{\mathcal G}_\sigma(y),
\qquad u,v\in\mathbb R^D.
\]
For submanifolds of arbitrary codimension, following Theorem~\ref{th:general-manifold} we define
\[
\widetilde{\Pi}_y(u,v)
=
\nabla_{\widetilde P_T(y)u}\widetilde P_T(y)\widetilde P_T(y)v,
\qquad u,v\in\mathbb R^D.
\]
\begin{corollary}\label{co:final-bound}
Assume the conditions of Corollary~\ref{co:rates-for-log-density-derivatives}. Then, for each fixed $0<\sigma\le\sigma_0$, the following hold. First, uniformly over $y\in T(\tau-\varepsilon)$,
\begin{align}\label{eq:tangent-sample-error}
 \sin\{\Theta(\widetilde{T}_{y}M,T_{\pi(y)}M)\} = O(\sigma^2) +O_P\big(\big(\frac{\log N}{N}\big)^{\frac{2}{D+8}}\big).
\end{align}
In the hypersurface case, uniformly over $y\in U_\sigma$,
\[
\|\widetilde{\Pi}_y-\Pi_{\pi(y)}\|_{\mathrm{op}} = O(\|v_y\|+\sigma) + O_P\big(\big(\frac{\log N}{N}\big)^{\frac{2}{D+8}}\big),
\]
In the arbitrary codimension setting, uniformly over $y\in T(\tau-\varepsilon)$,
\[
\|\widetilde{\Pi}_y-\Pi_{\pi(y)}\|_{\mathrm{op}} = O(\|v_y\|+\sigma^2)+O_P\big(\big(\frac{\log N}{N}\big)^{\frac{2}{D+10}}\big).
\]
In all three displays, the stochastic constants may depend on $\sigma$.
\end{corollary}
\begin{proof}[Sketch]
Fix $0<\sigma\le \sigma_0$. For the tangent-space bound, Corollary~\ref{co:noisy-hessian-gap} shows that, uniformly on $T(\tau-\varepsilon)$, $\mathcal H_\sigma(y)$ has a spectral gap of order $\sigma^{-2}$ separating its largest $d$ eigenvalues from the remaining $D-d$ eigenvalues. Therefore, for each fixed $\sigma$, Davis--Kahan's $\sin\Theta$ theorem applied to $\widetilde{\mathcal H}_\sigma(y)$ and $\mathcal H_\sigma(y)$ yields
\[
\sin\{ \Theta\big(\widetilde T_yM, \widehat T_yM\big)\}
\le
C\|\widetilde{\mathcal H}_\sigma(y)-\mathcal H_\sigma(y)\|_{\mathrm{op}}.
\]
This Hessian error is controlled by Corollary~\ref{co:rates-for-log-density-derivatives}, and combining this with the population estimate in Theorem~\ref{th:tangent-space} and the triangle inequality gives \eqref{eq:tangent-sample-error}.

In the hypersurface case, by the lower bound assumption on $\|\mathcal G_\sigma(y)\|$ over $U_\sigma$, Corollary~\ref{co:rates-for-log-density-derivatives} implies that $\|\widetilde{\mathcal G}_\sigma(y)\|$ remains bounded away from zero with high probability on $U_\sigma$. 
Therefore, on the region where $\|\mathcal G_\sigma\|,\|\widetilde{\mathcal G}_\sigma\|\ge c_0/2$,
the map 
\[
(\mathcal G,\mathcal H,P)\mapsto -\|\mathcal G\|^{-2} \langle \mathcal HPu,Pv\rangle \mathcal G
\]
is locally Lipschitz for fixed $\sigma$. It follows that the estimation error can be controlled by 
\[
\|\widetilde{\mathcal G}_\sigma-\mathcal G_\sigma\|,
\qquad
\|\widetilde{\mathcal H}_\sigma-\mathcal H_\sigma\|_{\mathrm{op}},
\qquad
\|\widetilde P_T-P_T\|_{\mathrm{op}}.
\]
The first two terms are bounded by Corollary~\ref{co:rates-for-log-density-derivatives}, and the last one is bounded by \eqref{eq:tangent-sample-error}. Combining this with the population bound from Theorem~\ref{th:hypersurface} gives the stated result.

In arbitrary codimension, the comparison between $\nabla \widetilde P_T(y)$ and $\nabla \widehat P_T(y)$ is a direct analogue of the Riesz-projector comparison carried out in the proof of Theorem~\ref{th:general-manifold}. By that argument, the difference of the projector derivatives is controlled by
\[
\|\widetilde{\mathcal H}_\sigma-\mathcal H_\sigma\|_{\mathrm{op}},
\qquad
\|\nabla \widetilde{\mathcal H}_\sigma-\nabla \mathcal H_\sigma\|_{\mathrm{op}},
\qquad
\|\nabla \mathcal H_\sigma\|_{\mathrm{op}}.
\]
For fixed $\sigma$, the last quantity is bounded on $T(\tau-\varepsilon)$, while the first two terms are controlled by Corollary~\ref{co:rates-for-log-density-derivatives}. Among them, the second term gives the dominated order $O_P\big(({\log N}/{N})^{\frac{2}{D+10}}\big)$. The ambient extension in \eqref{eq:pi-general-manifold} introduces additional factors of $\widetilde P_T(y)$ and $\widehat P_T(y)$, which are controlled by \eqref{eq:tangent-sample-error}. Combining this with Theorem~\ref{th:general-manifold} proves the claim.

\end{proof}
A notable feature of these rates is that they depend on the ambient dimension $D$, rather than the intrinsic dimension $d$.
This is inherent to the present estimator class. For each fixed $\sigma>0$, the target $P_\sigma = P_M*\mathcal N(0,\sigma^2 I_D)$ is a smooth bounded density on $\mathbb R^D$, and our sample-level analysis is based on estimating this ambient-space density and its derivatives by KDE-type methods.
Thus the resulting rates reflect the classical nonparametric difficulty of estimating a $D$-dimensional density; see, for example, \citet{goldenshluger2014adaptive}.
By contrast, the manifold structure enters primarily through the population-level geometric identities developed earlier. Obtaining sample-level bounds that adapt to the intrinsic dimension would likely require a different estimator class, one that more directly exploits the manifold structure.

\begin{remark}[Alternative estimators]
The sample-level arguments above use the kernel estimator only through uniform control of the estimated log-density derivatives. Accordingly, the same plug-in geometric estimators apply to any alternative estimators of
\[
\mathcal G_\sigma=\nabla\log P_\sigma,\qquad
\mathcal H_\sigma=\nabla^2\log P_\sigma,\qquad
\mathcal T_\sigma=\nabla^3\log P_\sigma
\]
that satisfy analogous bounds on the relevant region.

Possible alternatives include direct score estimators in nonparametric function classes, score-matching estimators, and smooth parametric models such as neural
score networks; see, for example, \citet{Hyvrinen2005EstimationON,Sriperumbudur2017,li2018gradient,Song2019SlicedSM}.
\end{remark}

\section{Numerical Experiments}\label{sec:numerical}
In this section we empirically evaluate the proposed estimators for tangent spaces and curvature on simulated datasets. Our theoretical analysis in previous sections shows that the geometry of the underlying manifold can be recovered from the derivatives of the manifold-generated noisy density $P_\sigma$, with error bounds that depend explicitly on the noise level $\sigma$ and the sample size $N$. The numerical experiments complement these results by studying how the estimation error scales with $\sigma$ and $N$, and by comparing our method with existing approaches for manifold geometry estimation.

We focus on manifolds for which tangent spaces and second fundamental forms are available in closed form. In particular, we consider a two-dimensional torus embedded in $\mathbb{R}^3$ (a hypersurface) and the two-dimensional Clifford torus embedded in $\mathbb{R}^4$ (a submanifold of codimension two). These examples exhibit nontrivial and spatially varying curvature while remaining simple enough to admit accurate numerical ground truth. In Section~\ref{sec:exp-setup} we describe the experimental setup. Sections~\ref{sec:exp-tangent}--\ref{sec:exp-curv-clifford} report tangent space and curvature estimation results on the torus and the Clifford torus, and Section~\ref{sec:exp-geodesics} investigates geodesics under the gradient-based density-induced metric.

\subsection{Experimental Setup}
\label{sec:exp-setup}
We first describe the data generation procedure, the estimator construction, the error metrics, and the baseline and hyperparameter choices.
\subsubsection{Data Generation} Let $M \subset \mathbb{R}^D$ be one of the manifolds described above, endowed with its induced Riemannian metric and volume measure. For each choice of noise level $\sigma$ and sample size $N$, we generate noisy data by first sampling i.i.d. points $x_1,\dots,x_N$ uniformly on $M$, and then adding independent Gaussian noise
\[
  y_i = x_i + \xi_i, \qquad \xi_i \sim \mathcal{N}(0,\sigma^2 I_D).
\]
The resulting observations $\mathcal{Y} = \{y_i\}_{i=1}^N $ are i.i.d. samples from the noisy distribution $P_\sigma$, consistent with the population-level analysis of Section~\ref{sec:density-and-derivatives}.

To isolate the effect of the noise level $\sigma$ and the sample size $N$ from the additional error due to the normal offset $\|v_y\|$, we evaluate the geometric estimators at points in a thin tubular neighborhood of $M$. Concretely, we generate an independent set of evaluation points $\mathcal W = \{w_i\}_{i=1}^{\bar{N}}$ such that
\[
  c_1 \sigma^2 \log(1/\sigma)
  \leq
  d(M,w_i)
  \leq
  c_2 \sigma^2 \log(1/\sigma),
  \qquad i=1,\ldots,\bar{N},
\]
for fixed constants $0 < c_1 < c_2$. Throughout the experiments we take $c_1 = 1/2$, $c_2 = 2$, and $\bar{N} = 50$.
These points can be viewed as proxies for the output of a manifold reconstruction step. This choice is motivated by existing manifold fitting results that achieve Hausdorff accuracy on the order of $O\big(\sigma^2\log(1/\sigma)\big)$ in the small-noise regime (see, e.g., \citealp{Genovese2014Ridge, yao2023manifold}), so that the $\|v_y\|$-dependent terms in our bounds decay at a comparable scale as $\sigma\to 0$.

\subsubsection{Kernel Density and Derivative Estimators}
At the sample level we follow Section~\ref{sec:sample-level-estimator} and estimate the noisy density $P_\sigma$ and its derivatives using kernel methods. In the numerical experiments we use the standard Gaussian kernel
\[
  K(u) = (2\pi)^{-D/2} \exp(-\tfrac{1}{2}\|u\|^2),
  \qquad u \in \mathbb{R}^D.
\]
As suggested by Lemma~\ref{lem:kernel-density-estimator}, for each estimator that involves derivatives of order $m \in \{0,1,2,3\}$ of $P_\sigma$, we choose the bandwidth according to
\[
  h_m
  = c \big(\frac{\log N}{N}\big)^{1/(D + 4 + 2m)}.
\]
In the experiments, we choose the constants $c$ in a simple, data-dependent way following Scott’s rule \citep{Scott1992MultivariateDE}. Specifically, we take $c$ to be the average standard deviation of the sample $\mathcal{Y}$ along the largest $d$ principal components. For each geometric estimator we use the bandwidth associated with the highest-order derivative that it requires:
\begin{itemize}
  \item the gradient-based metric in Section~\ref{sec:exp-geodesics} uses $h_1$;
  \item the tangent space estimator and the hypersurface curvature estimator (which rely on the Hessian of the log-density) use $h_2$;
  \item the general-codimension curvature estimator (which additionally uses the third derivative of the log-density) uses $h_3$.
\end{itemize}

\subsubsection{Error Metrics}
For a point $w \in \mathcal W$ we denote by $T_{\pi(w)} M$ the true tangent space and by $\widetilde{T}_{w} M$ its estimate. The error in tangent space estimation is quantified by the sine of the largest principal angle between the two subspaces,
\[
  \mathrm{err}_{\mathrm{tan}}(w)
  := \sin \Theta\{\big(T_{\pi(w)} M, \widetilde{T}_{w} M\big)\}.
\]
For curvature we focus on the mean curvature vector $H_x \in T_x^\perp M$, which is a normal field on $M$ obtained by averaging the second fundamental form over an orthonormal basis of $T_x M$. Denoting by $\widetilde{H}_w$ the estimator derived from the estimated second fundamental form, we define the pointwise curvature error as
\[
  \mathrm{err}_{\mathrm{curv}}(w)
  := \|H_{\pi(w)} - \widetilde{H}_w\|_2.
\]
For each configuration of $(\sigma,N)$ we summarize the distribution of these errors over the evaluation set $\mathcal W$ by box plots.

\subsubsection{Baselines and Hyperparameter Selection}
Our main comparisons are with the following baseline methods.
\begin{itemize}
  \item For tangent space estimation, we consider LPCA as a classical baseline and the manifold diffusion geometry estimator of JI24~\citep{Jones2025ManifoldDG}. In both cases the tangent space at $\pi(w)$ is estimated from observations in a neighborhood of $w$. For LPCA we use a kernel-weighted covariance matrix based on the same Gaussian kernel as above, with bandwidth
    \[
      h_{\mathrm{LPCA}}
      = c N^{-1/(d+4)},
    \]
    where the constant $c$ is the same as in our estimator so that the effective local scale is comparable. JI24 requires little additional tuning, and we follow their standard implementation without additional tuning.

  \item For curvature estimation, we compare our method with JI24 and with the Weingarten map estimator of CY19~\citep{Cao2019EfficientWM}, which fits the shape operator from local covariance information. The optimal bandwidth proposed in CY19 is of the order $O(N^{-1/(d+4)})$, and we therefore set
    \[
      h_{\mathrm{WM}}
      = c N^{-1/(d+4)},
    \]
    again using the same constant $c$ as above. The JI24 curvature estimator uses its default implementation.
\end{itemize}
In each experiment we first carry out a preliminary comparison that includes all available baselines. Based on the median error over the evaluation set $\mathcal W$, we then focus on the stronger baseline in the more detailed comparisons.

\subsubsection{Parameter Grids}
To investigate the asymptotic behavior in $\sigma$ and $N$, we consider two families of configurations. In the first we fix $N = 10^4$ and vary
\[
   \sigma \in \{0.01, 0.025, 0.05, 0.075, 0.1, 0.5, 1.0\}.
\]
In the second we fix $\sigma = 0.05$ and vary
\[
   N \in \{300, 3000, 30000, 300000\}.
\]
These grids are used in the scaling experiments for our proposed method (Figures~\ref{fig:tangent-torus}, \ref{fig:curv-torus}, and \ref{fig:curv-clifford}). In the baseline comparisons we additionally consider smaller grids, fixing $N = 3000$ and varying $\sigma \in \{0.01, 0.1, 0.5\}$, or fixing $\sigma = 0.05$ and varying $N \in \{300, 1000, 3000\}$, as specified below.

\begin{figure}[!t]
    \centering
    \includegraphics[width=0.9\linewidth]{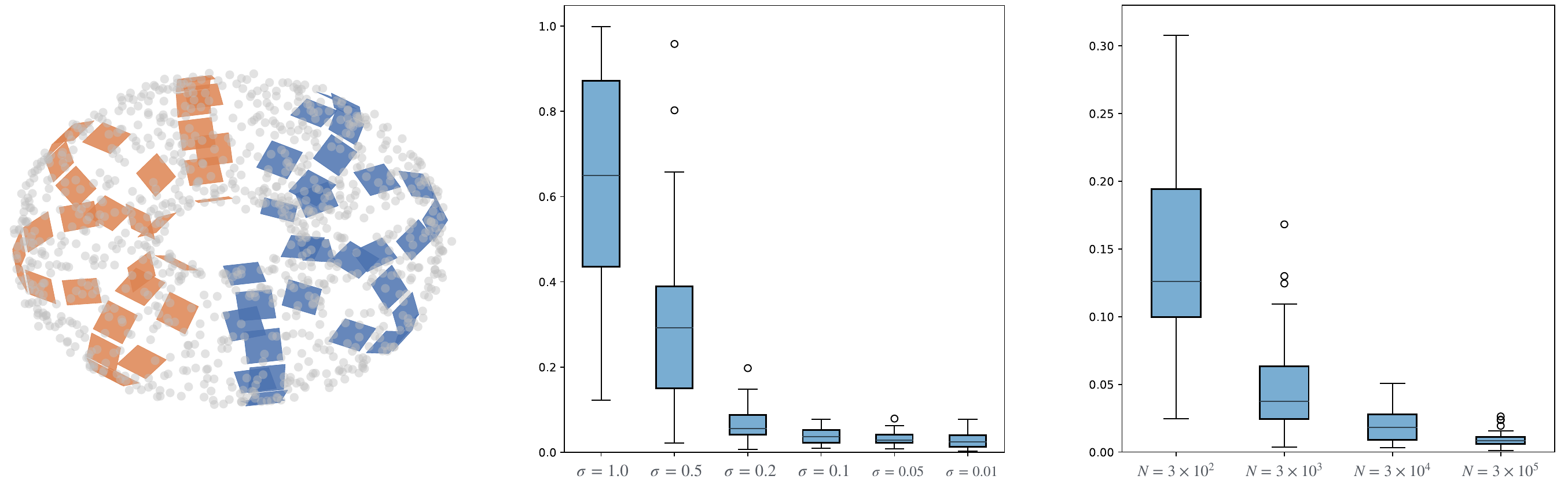}
    \caption{Tangent space estimation on a noisy torus. Left: noisy samples $\mathcal Y$ (gray), together with the true tangent planes (orange) and the estimated tangent planes (blue) at representative points on the torus for $N=1000$ and $\sigma=0.05$. Middle and right: box plots of the tangent space error $\mathrm{err}_{\mathrm{tan}}$ over the evaluation set $\mathcal W$ as functions of the noise level $\sigma$ for fixed $N = 10^4$ and of the sample size $N$ for fixed $\sigma = 0.05$, respectively.}
    \label{fig:tangent-torus}
\end{figure}

\begin{figure}[!t]
    \centering
    \includegraphics[width=0.9\linewidth]{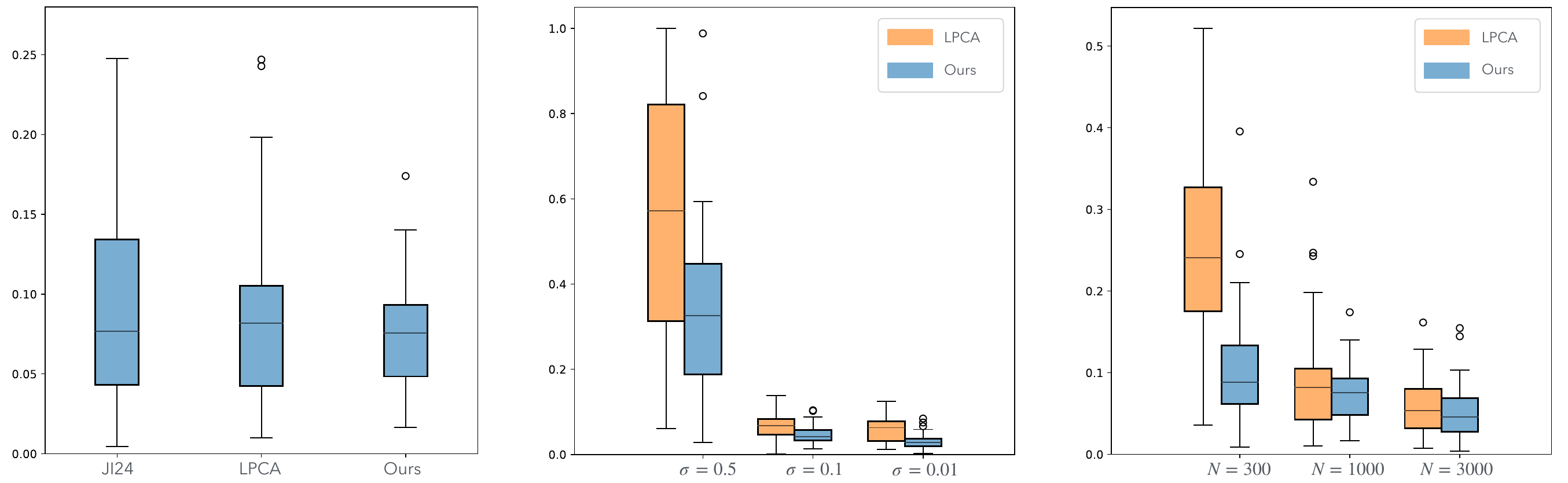}
    \caption{Comparison of tangent space estimators on the torus. Left: box plots of tangent space errors for $N = 1000$ and $\sigma = 0.05$, comparing LPCA, the diffusion-geometry estimator JI24, and the proposed Hessian-based estimator. Middle and right: errors of LPCA and of the proposed estimator as functions of the noise level $\sigma$ for fixed $N = 3000$ and of the sample size $N$ for fixed $\sigma = 0.05$, respectively.}
    \label{fig:tangent-comparison}
\end{figure}

\subsection{Tangent Space Estimation on the Torus}
\label{sec:exp-tangent}
We first study the tangent space estimator obtained from the Hessian of the log-density, as in Theorem~\ref{th:tangent-space}. For an evaluation point $w \in \mathcal W$ we compute the Hessian $\widetilde{\mathcal H}_\sigma(w)$ of the estimated log-density and define $\widetilde{T}_w M$ as the span of its $d$ largest eigenvectors.

The underlying manifold is a two-dimensional torus embedded in $\mathbb{R}^3$. We investigate the dependence of the tangent space error on both the noise level $\sigma$ and the sample size $N$ using the parameter grids described in Section~\ref{sec:exp-setup}. Figure~\ref{fig:tangent-torus} illustrates the experimental design and summarizes the quantitative results. For moderate noise levels $\sigma \leq 0.5$, the principal angle error remains well below $\sin 45^\circ$ across the torus. At the extreme noise level $\sigma = 1$, the error begins to exceed this threshold, indicating that the manifold is no longer clearly identifiable from the noisy observations. As $N$ increases the error decreases rapidly, and the estimator remains effective even at relatively small sample sizes such as $N = 300$.

We next compare our estimator with the LPCA and JI24 tangent space estimators. For a representative configuration $N = 1000$, $\sigma = 0.05$, Figure~\ref{fig:tangent-comparison} (left) displays the error distributions of all three methods. In this setting LPCA yields noticeably smaller errors than JI24, and we therefore focus on LPCA in the more detailed comparisons. The middle and right panels of Figure~\ref{fig:tangent-comparison} report the errors of LPCA and the proposed estimator as functions of $\sigma$ and $N$, respectively. Across these regimes our estimator consistently yields smaller errors than LPCA, especially at smaller sample sizes and for larger noise levels, in line with the higher-order dependence on $\sigma$ predicted by Theorem~\ref{th:tangent-space}.

\begin{figure}[!t]
    \centering
    \includegraphics[width=0.9\linewidth]{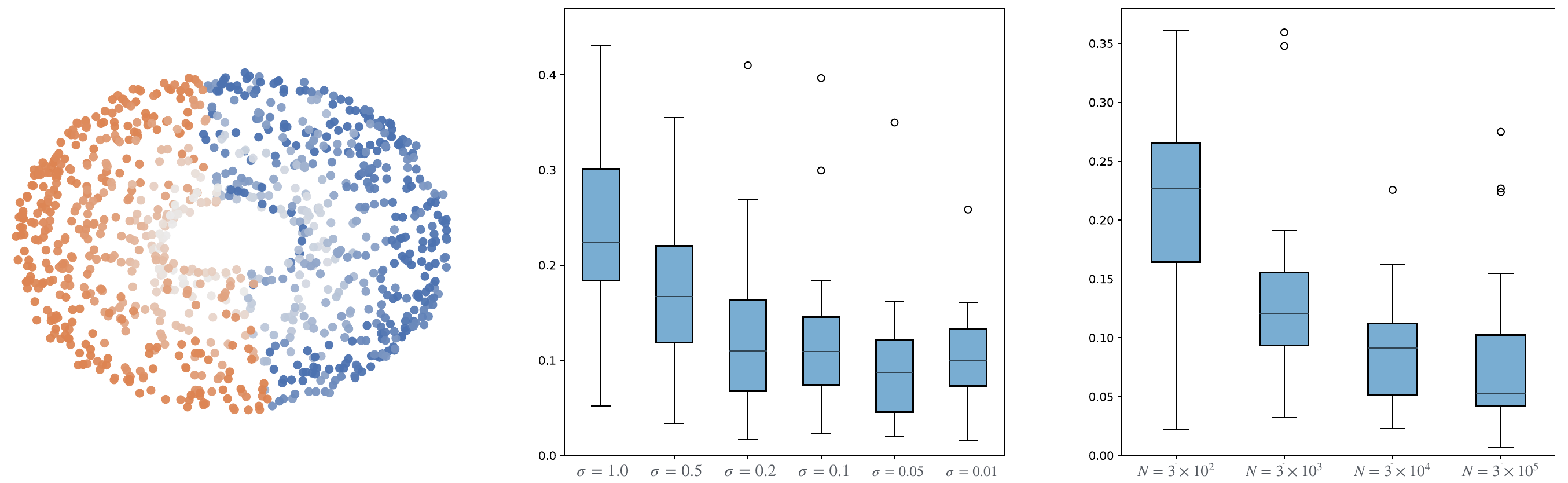}
    \caption{Mean curvature estimation on the torus. Left: norm of the true mean curvature vectors (orange) and of the estimated mean curvature vectors (blue) on the torus for $N=1000$ and $\sigma=0.05$; darker colors indicate larger norm. Middle and right: box plots of the curvature error $\mathrm{err}_{\mathrm{curv}}$ over the evaluation set $\mathcal W$ as functions of the noise level $\sigma$ for fixed $N = 10^4$ and of the sample size $N$ for fixed $\sigma = 0.05$, respectively.}
    \label{fig:curv-torus}
\end{figure}

\begin{figure}[!t]
    \centering
    \includegraphics[width=0.9\linewidth]{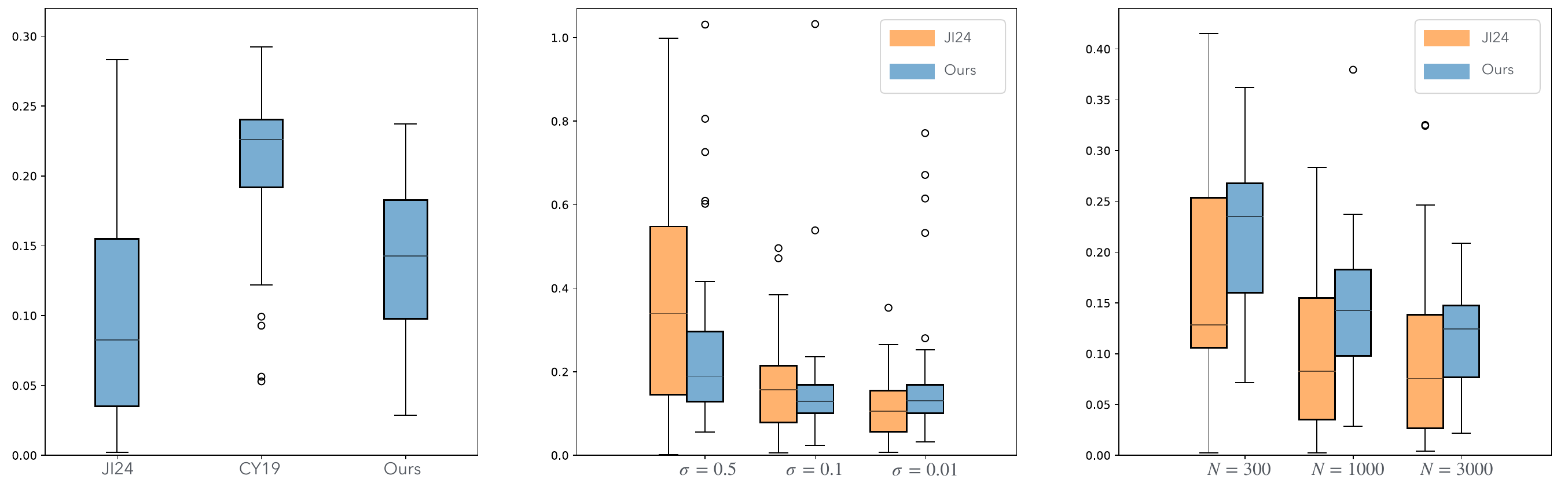}
    \caption{Comparison of mean curvature estimators on the torus. Left: box plots of curvature errors for $N = 1000$ and $\sigma = 0.05$, comparing the Weingarten map estimator of CY19, the diffusion-geometry estimator JI24, and the proposed hypersurface estimator. Middle and right: curvature errors of JI24 and of the proposed estimator as functions of the noise level $\sigma$ for fixed $N = 3000$ and of the sample size $N$ for fixed $\sigma = 0.05$, respectively.}
    \label{fig:curv-torus-comp}
\end{figure}

\subsection{Curvature Estimation on the Torus}
\label{sec:exp-curv-torus}
We now turn to curvature estimation on the same two-dimensional torus in $\mathbb{R}^3$, treating it as a hypersurface. Using the hypersurface second fundamental form estimator from Theorem~\ref{th:hypersurface}, we obtain an estimate $\widetilde{\Pi}$ at each evaluation point and compute the corresponding mean curvature vector $\widetilde{H}$ by averaging $\widetilde{\Pi}$ over an orthonormal basis of the estimated tangent space. This yields the curvature error $\operatorname{err}_{\mathrm{curv}}$ defined in Section~\ref{sec:exp-setup}.

Figure~\ref{fig:curv-torus} visualizes the norm of the mean curvature vector on the torus and summarizes the dependence of the curvature error on $\sigma$ and $N$. The estimator is accurate for moderate noise levels and improves with increasing sample size, but curvature estimation is noticeably less stable than tangent space estimation: the variance of $\operatorname{err}_{\mathrm{curv}}$ is larger, and a few outliers with relatively large errors are present. This is consistent with the theoretical bounds, which scale as $O(\sigma)$ with respect to the noise level for the hypersurface curvature estimator (Theorem~\ref{th:hypersurface}), compared with $O(\sigma^2)$ for the tangent space estimator (Theorem~\ref{th:tangent-space}).

We compare the proposed method with the Weingarten map estimator of CY19 and the curvature estimator of JI24. For a representative configuration $N = 1000$, $\sigma = 0.05$, Figure~\ref{fig:curv-torus-comp} (left) shows that CY19 performs substantially worse than both JI24 and our estimator. We therefore take JI24 as the main curvature baseline in the subsequent plots. The middle and right panels of Figure~\ref{fig:curv-torus-comp} display the curvature errors of JI24 and of the proposed estimator as functions of $N$ and $\sigma$, respectively. For the small-noise level $\sigma = 0.05$ our estimator has a smaller variance and a smaller maximum error than JI24 but a slightly larger median error, suggesting that JI24 may converge more rapidly in the asymptotically small-noise regime. On the other hand, as the noise level increases beyond $\sigma \approx 0.1$ our estimator becomes substantially more robust than JI24, maintaining relatively small errors while the performance of JI24 deteriorates.

\begin{figure}[!t]
    \centering
    \includegraphics[width=0.9\linewidth]{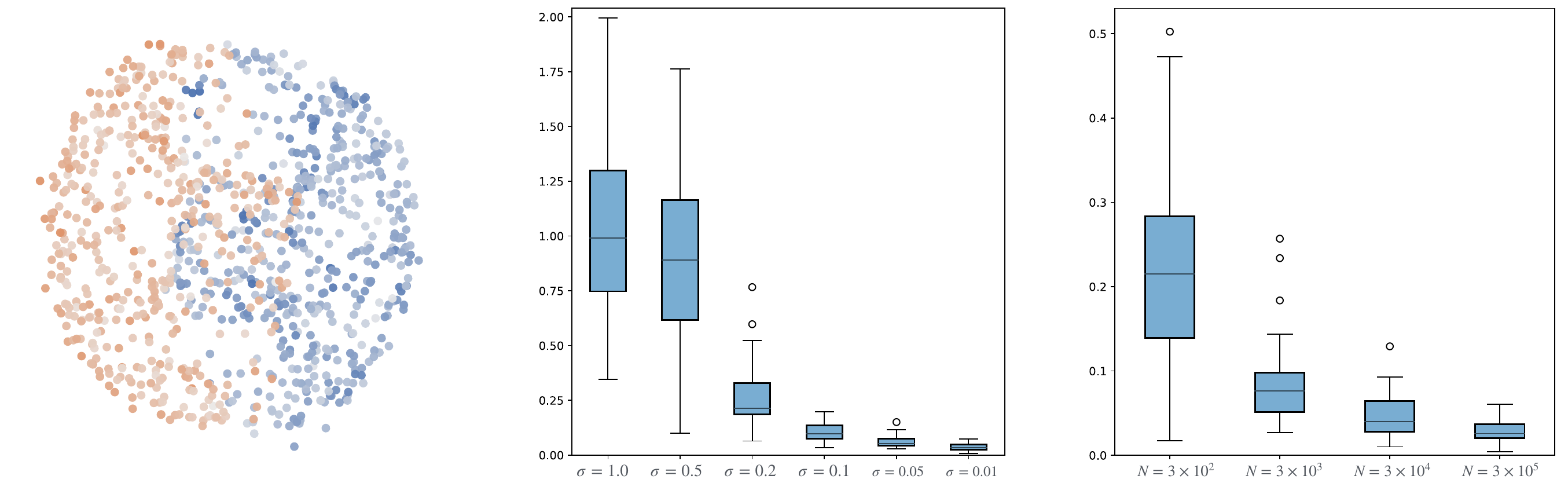}
    \caption{Mean curvature estimation on the Clifford torus in $\mathbb{R}^4$. Left: norm of the true mean curvature vectors (orange) and of the estimated mean curvature vectors (blue) on the Clifford torus for $N=1000$ and $\sigma=0.05$; darker colors indicate larger norm. Middle and right: box plots of the curvature error $\mathrm{err}_{\mathrm{curv}}$ over the evaluation set $\mathcal W$ as functions of the noise level $\sigma$ for fixed $N = 10^4$ and of the sample size $N$ for fixed $\sigma = 0.05$, respectively.}
    \label{fig:curv-clifford}
\end{figure}

\begin{figure}[!t]
    \centering
    \includegraphics[width=0.9\linewidth]{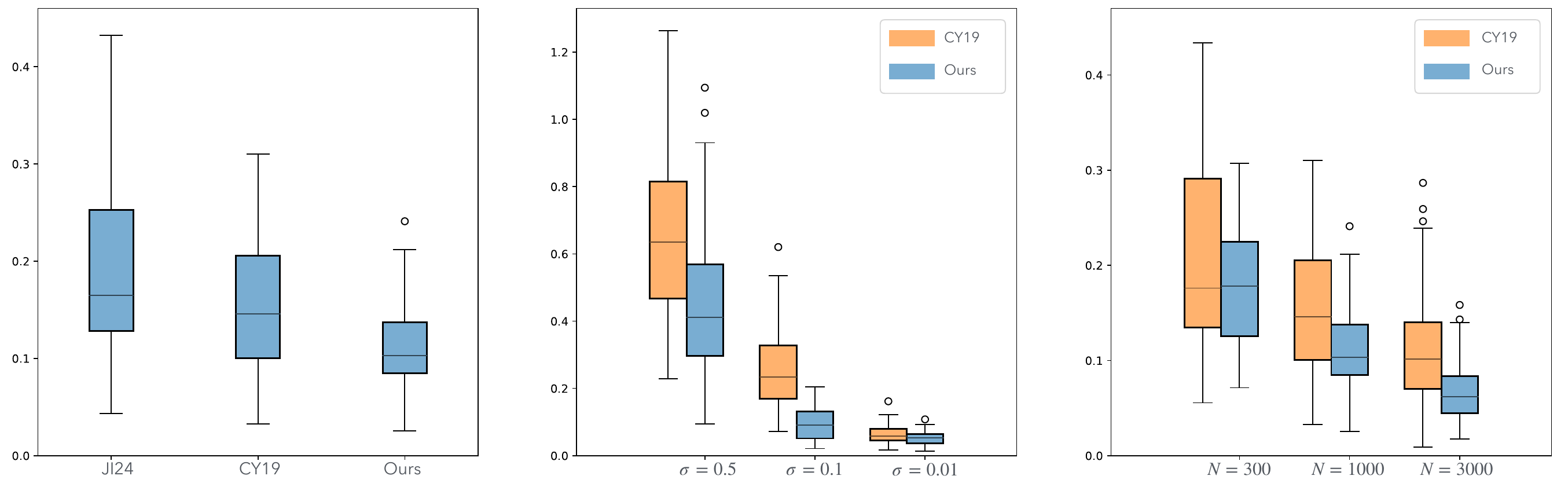}
    \caption{Comparison of mean curvature estimators on the Clifford torus. Left: box plots of curvature errors for $N = 1000$ and $\sigma = 0.05$, comparing CY19, JI24, and the proposed general-codimension estimator. Middle and right: curvature errors of CY19 and of the proposed estimator as functions of the noise level $\sigma$ for fixed $N = 3000$ and of the sample size $N$ for fixed $\sigma = 0.05$, respectively.}
    \label{fig:curv-clifford-comp}
\end{figure}

\subsection{Curvature Estimation on the Clifford Torus}
\label{sec:exp-curv-clifford}
To evaluate the estimator for the second fundamental form of submanifolds with general codimension, we consider the two-dimensional Clifford torus embedded in $\mathbb{R}^4$. In contrast to the hypersurface case, the normal space now has dimension two, and the second fundamental form has nontrivial components in multiple normal directions, making this a more challenging setting. We use the general-codimension estimator of Theorem~\ref{th:general-manifold} to obtain an estimate $\widetilde{\Pi}$ at each evaluation point on the Clifford torus, and define the estimated mean curvature vector $\widetilde{H}$ and error $\operatorname{err}_{\mathrm{curv}}$ as in Section~\ref{sec:exp-curv-torus}.

Figure~\ref{fig:curv-clifford} visualizes the norm of the mean curvature on the Clifford torus and summarizes the dependence of the curvature error on $\sigma$ and $N$. The proposed estimator achieves accurate mean curvature estimates across a wide range of noise levels and sample sizes. Although the general-codimension estimator is more complex than the hypersurface counterpart, it exhibits significantly better convergence behavior, which is consistent with its $O(\sigma^2)$ convergence rate.

Figure~\ref{fig:curv-clifford-comp} compares the proposed estimator with CY19 and JI24. While JI24 is more accurate than the local fitting method CY19 on two-dimensional hypersurfaces, it does not perform consistently across the parameter regimes considered in this higher-codimension setting. Meanwhile, the proposed estimator consistently outperforms CY19, with smaller errors and lower variance across the parameter regimes considered.

\begin{figure}[!t]
    \centering
    \includegraphics[width=0.9\linewidth]{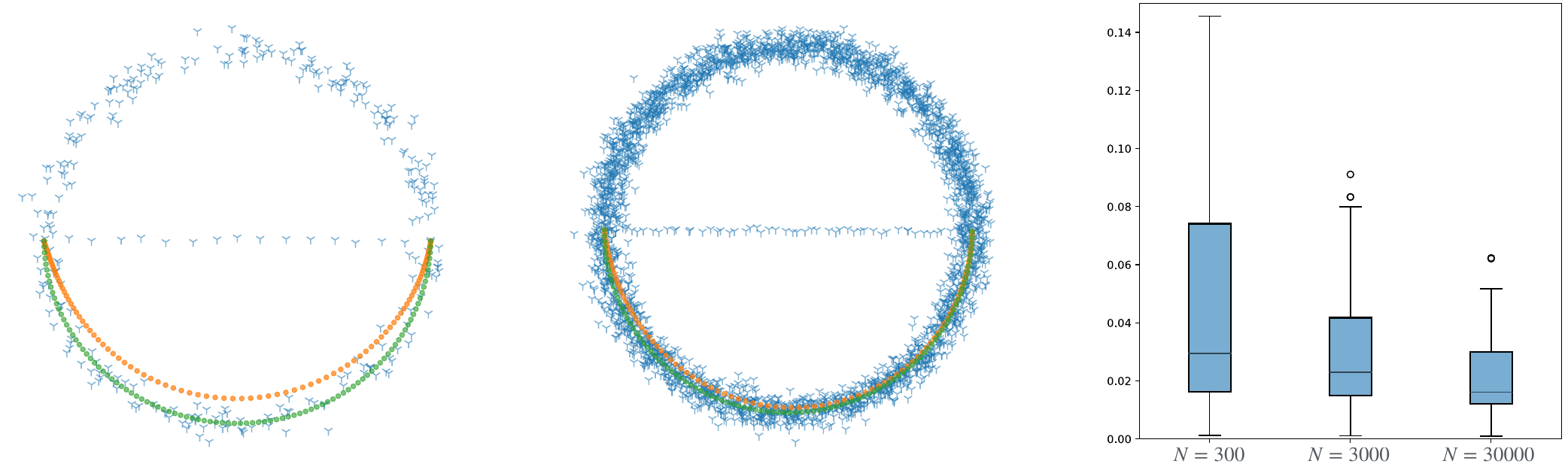}
    \caption{Geodesics under the gradient-based metric in the presence of shortcut perturbations. Left and middle: true geodesic (green) on the unit circle $S^1$ and estimated $g_{P_\sigma}$-geodesics (orange) between two fixed endpoints for fixed $\sigma=0.05$, $N = 300$ and $N = 3000$, respectively. Right: Hausdorff distance between the computed geodesics and the true geodesic on $S^1$ as a function of the sample size $N$ for fixed $\sigma=0.05$.}
    \label{fig:perturbed geodesics}
\end{figure}

\subsection{Geodesics of the Gradient-based Metric}
\label{sec:exp-geodesics}
Finally, we investigate the robustness of the density-induced metrics introduced in Section~\ref{sec:metric-change}. In particular, we focus on the gradient-based degenerate metric. Recall that for a noisy distribution $P_\sigma$ we define the degenerate metric,
\[
  g_{P_\sigma} := d(\log P_\sigma)\otimes d(\log P_\sigma).
\]
Theorem~\ref{th:metric-hypersurface} shows that, for a manifold-generated density, intrinsic geodesics on $M$ satisfy the ambient acceleration equation associated with $g_{P_\sigma}$ up to an $O(\sigma)$ error.

To test the stability of geodesics under perturbations of the data, we introduce an artificial shortcut perturbation, analogous to the numerical shortcut experiment of \citet{Taupin2025FermatDA}. Fix $\sigma = 0.05$ and let $N \in \{300, 3000, 30000\}$. For each $N$ and each repetition, we first randomly select two evaluation points $w_1, w_2 \in \mathcal W$. We then perturb the data by inserting $\sqrt{N}$ additional points along the straight line segment connecting $w_1$ and $w_2$ in the ambient space, and we further perturb these shortcut points by independent Gaussian noise with the same variance $\sigma = 0.05$. This creates an artificial density bridge between the two endpoints.

For the true geodesic, we project $w_1$ and $w_2$ onto the manifold, setting $x_1 = \pi(w_1)$ and $x_2 = \pi(w_2)$, and compute the corresponding intrinsic geodesic between $x_1$ and $x_2$ analytically. For the estimated geodesic, we start from an initial discretized path between $w_1$ and $w_2$ with $100$ intermediate points and minimize its $g_{P_\sigma}$-length using gradient-based optimization while keeping the endpoints fixed. We then compare the resulting path with the true geodesic by computing their Hausdorff distance in the ambient Euclidean space. For each $N$ we repeat this procedure over $50$ random choices of $(w_1,w_2) \in \mathcal W \times \mathcal W$ and summarize the resulting Hausdorff distances by box plots.

Figure~\ref{fig:perturbed geodesics} shows representative examples of the computed geodesics for $N = 300$ (left) and $N = 3000$ (middle), together with box plots of the Hausdorff distance as a function of $N$ (right). The distances decrease as the sample size grows and remain small even in the presence of the shortcut points, indicating that the gradient-based metric $g_{P_\sigma}$ keeps geodesics close to the underlying manifold and is robust to this type of perturbation.

\section{Proofs}\label{sec:proof}
This section collects proofs of the main results in Section~\ref{sec:density-and-derivatives}--\ref{sec:sample-level-estimator}.

\subsection{Proofs for Section~\ref{sec:density-and-derivatives}}
\label{sec:pf-density-derivative}
We first prove the results stated in Section~\ref{sec:density-and-derivatives}.

\subsubsection{Proof of Theorem~\ref{th:probability-noisy}}\label{sec:pf-probability noisy}
\begin{proof}
Write $x_0:=\pi(y)$. Recall that $A_y:=I_{T_{\pi(y)}M}-\langle v_y,\Pi_{x_0}\rangle$. Using $\|\Pi_{x}\|_{\mathrm{op}}\le \tau^{-1}$ (Lemma~\ref{lem:second-fundamental-form-bound}) and $\|v_y\|\le \tau-\varepsilon$,
\[
\lambda_{\min}(A_y)\ \ge\ 1-\|v_y\|/\tau\ \ge\ \varepsilon/\tau\ =:\ c_\varepsilon>0,
\]
so $A_y$ is uniformly positive definite on $T(\tau-\varepsilon)$. 

Since $M$ is compact, we may choose a finite smooth atlas and, on each chart, a smooth orthonormal frame
$e_1(x),\dots,e_d(x)$ for $TM$. Using this frame, we identify each tangent space $T_xM$ locally with
$\mathbb R^d$ by
\[
u=(u^1,\dots,u^d)\in\mathbb R^d 
\quad\longleftrightarrow\quad
\sum_{i=1}^d u^i e_i(x)\in T_xM.
\]
We work in exponential coordinates at $ x_0$ and write $x_0(u):=\exp_{x_0}(u)$ for $u\in \mathbb R^d$. Choose $r_0\in(0,\mathrm{inj}_M)$ sufficiently small so that, uniformly in $x_0\in M$, the exponential chart is well defined on $B_d(0,r_0)$ and the expansions in Lemma~\ref{lem:exponential-map-expansion} and Lemma~\ref{lem:volume-element-expansion} hold there. Then in these coordinates
\[
\sqrt{\det g_{x_0}(u)}=1+Q_g(y,u),\qquad |Q_g(y,u)|\le C_g\|u\|^2, 
\]
\[
x_0(u)=x_0+u+\tfrac12\Pi_{x_0}(u,u)+Q_e(y,u),\qquad \|Q_e(y,u)\|\le C_e\|u\|^3.
\] 
Moreover, by the regularity of exponential map, metric coefficients, and projection $\pi$, we have
\[
|\partial_y^\alpha Q_g(y,u)|\le C_\alpha \|u\|^2,
\quad
\|\partial_y^\alpha Q_e(y,u)\|\le C_\alpha \|u\|^3,\qquad |\alpha|\le3.
\]
Expanding $\|y-x_0(u)\|^2$ after inserting the expression of $x_0(u)$, and using the orthogonality $u\perp \Pi_{x_0}(u,u)$, one obtains
\[
\|y-x_0(u)\|^2 = \|v_y\|^2 + u^\top A_y u + R(y,u),\qquad |R(y,u)|\le C\big(\|v_y\|\|u\|^3+\|u\|^4\big).
\]
Since $R(y,u)$ is obtained by combining the cubic remainder $Q_e(y,u)$ with bounded coefficients that depend $C^3$-smoothly on $y$, differentiating in $y$ preserves the cubic order:
\[
|\partial_y^\alpha R(y,u)|\le C_\alpha\|u\|^3, \qquad |\alpha|\le3.
\]
In particular, on the compact set $T(\tau-\varepsilon)\times B_d(0,r)$ we can take the above constants uniformly. Indeed, by shrinking $r_0$ if necessary, we also have 
\[
|R(y,u)|\le \frac{c_\varepsilon}{4}\|u\|^2, \qquad \|u\|\le r_0.
\]
Fix any $r\in(0,r_0)$. Let $\rho\in C^k(\mathbb R^d)$ satisfy
\[
0\le \rho \le 1,\qquad
\rho(u)=1 \text{ for } \|u\|\le r/2,\qquad
\rho(u)=0 \text{ for } \|u\|\ge r.
\]
Define a cutoff function on $M$ by
\[
\chi_y(x)= \rho(\exp_{x_0}^{-1}(x)) \text{ for } x\in \exp_{x_0}(B_d(0,r_0)),\qquad \chi_y(x)= 0 \text{ for } x\notin \exp_{x_0}( B_d(0,r_0)).
\]
By the regularity of the local inverse exponential map and the projection map, $(y,x)\mapsto \rho(\exp_{\pi(y)}^{-1}(x))$ is $C^{k-2}$ on $\{(y,x): y\in T(\tau-\varepsilon),\ d_M(x,\pi(y))< r_0\}$.
Moreover, since $\rho$ is compactly supported on $B_d(0,r)$ with $r<r_0$, the zero extension of $\chi_y(x)$ defines $C^{k-2}$ function on
$T(\tau-\varepsilon)\times M$.
By compactness of $T(\tau-\varepsilon)\times M$, we further have a uniform bound
\[
|\partial_y^\alpha \chi_y(x)|\le C_\alpha,\qquad |\alpha|\le 3.
\]
Using the cutoff function, we decompose
\[
P_\sigma(y)=\frac{1}{V_M(2\pi\sigma^2)^{\frac{D}{2}}}\big(I_1(y,\sigma)+I_2(y,\sigma)\big),
\]
where
\[
I_1(y,\sigma)
:=
\int_M \chi_y(x)\exp\big(-\frac{\|y-x\|^2}{2\sigma^2}\big)d\mu(x),
\]
\[
I_2(y,\sigma)
:=
\int_M (1-\chi_y(x))\exp\big(-\frac{\|y-x\|^2}{2\sigma^2}\big)d\mu(x).
\]
By construction, the support of $\chi_y$ is contained in $\exp_{x_0}(B_d(0,r_0))$, so $I_1$ can be written in exponential coordinates at $x_0$.

%I_1
Using the expansions in exponential coordinates and extending the integrand by zero outside $B_d(0,r_0)$, we obtain
\[
I_1(y,\sigma)
=
\int_{\mathbb R^d}
\chi_y(x_0(u))
\exp\big(-\frac{1}{2\sigma^2}\big(\|v_y\|^2 + u^\top A_y u + R(y,u)\big)\big)
\big(1+Q_g(y,u)\big)du.
\]
Make the change of variables $z = u/\sigma$, and we write $\eta_\sigma(z) = \rho(\sigma z)$. Then
\[
I_1(y,\sigma)
=
e^{-\|v_y\|^2/(2\sigma^2)}\sigma^d
\int_{\mathbb R^d}
\eta_\sigma(z)
e^{-\frac12 z^\top A_y z}
e^{-R(y,\sigma z)/(2\sigma^2)}
\big(1+Q_g(y,\sigma z)\big)dz.
\]
We decompose $I_1$ as
\[
I_1(y,\sigma)
=
e^{-\|v_y\|^2/(2\sigma^2)}\sigma^d
\big(E_0(y)+E_1(y,\sigma)+E_2(y,\sigma)\big),
\]
where
\[
E_0(y):=\int_{\mathbb R^d} e^{-\frac12 z^\top A_y z}dz
=\frac{(2\pi)^{d/2}}{\sqrt{\det A_y}},
\]
\[
E_1(y,\sigma)
:=
\int_{\mathbb R^d}
\eta_\sigma(z)e^{-\frac12 z^\top A_y z}
\big(
e^{-R(y,\sigma z)/(2\sigma^2)}\big(1+Q_g(y,\sigma z)\big)-1
\big)dz,
\]
\[
E_2(y,\sigma)
:=
\int_{\mathbb R^d}
(\eta_\sigma(z)-1)e^{-\frac12 z^\top A_y z}dz.
\]
We first estimate $E_2$. 
Since $A_y\ge c_\varepsilon I_{T_{\pi(y)}M}$ and $\eta_\sigma(z)-1\neq 0$ implies $\|z\|\ge r/(2\sigma)$, we have
\[
|E_2(y,\sigma)|
\le
\int_{\|z\|\ge r/(2\sigma)} e^{-(c_\varepsilon/2)\|z\|^2}dz
\le C e^{-c/\sigma^2}.
\]
To estimate derivatives, let $\phi(y,z):=\frac12 z^\top A_y z$. Since $\partial_y^\alpha A_y$ is uniformly bounded for $|\alpha|\le 3$,
\[
|\partial_y^\alpha \phi(y,z)|\le C_\alpha \|z\|^2, \qquad |\alpha|\le 3.
\]
By repeated application of the chain rule to $e^{-\frac12 z^\top A_y z}$,
\[
\Big|
\partial_y^\alpha e^{-\frac12 z^\top A_y z}
\Big|
\le
C_\alpha (1+\|z\|^{m_\alpha})e^{-(c_\varepsilon/2)\|z\|^2}.
\]
Therefore by $|\eta_\sigma-1|\le1$ and the dominated convergence,
\[
|\partial_y^\alpha E_2(y,\sigma)|
\le
C_\alpha\int_{\|z\|\ge r/(2\sigma)} (1+\|z\|^{m_\alpha})e^{-(c_\varepsilon/2)\|z\|^2}
\le
C_\alpha e^{-c/\sigma^2},
\quad |\alpha|\le 3.
\]
We next estimate $E_1$. Set
\[
a(y,z,\sigma):=\frac{R(y,\sigma z)}{2\sigma^2},
\]
then by the bounds of $R(y, u)$ we have 
\[
|a(y,z,\sigma)|\le C(\sigma\|v_y\|\|z\|^3+\sigma^2\|z\|^4),
\qquad
|\partial_y^\alpha a(y,z,\sigma)|\le C_\alpha\sigma\|z\|^3,
\]
and
\[
|a(y,z,\sigma)| \le \frac{c_\varepsilon}{8}\|z\|^2.
\]
We write
\[
G(y,z,\sigma):=
e^{-\phi(y,z)}
\big(
e^{-a(y,z,\sigma)}\big(1+Q_g(y,\sigma z)\big)-1
\big).
\]
First, by 
\[
|e^{-a}(1+Q_g)-1| = |(e^{-a}-1)(1+Q_g)+Q_g|\le|a|e^{|a|}|1+Q_g|+|Q_g|,
\]
we have 
\[
|e^{-a(y,z,\sigma)}(1+Q_g(y,\sigma z))-1| \le C(\sigma\|v_y\|+\sigma^2)(1+\|z\|^m)e^{(c_\varepsilon/8)\|z\|^2}.
\]
Combining the bound of $\phi(y,z)$, we get
\[
|G(y,z,\sigma)|
\le
C(\sigma\|v_y\|+\sigma^2)(1+\|z\|^m)e^{-(3c_\varepsilon/8)\|z\|^2}.
\]
Now consider the derivatives of $G(y,z,\sigma)$. By Leibniz’ rule, every term arising in $\partial^\alpha_y G$ contains either $e^{-a}(1+Q_g)-1$, a derivative of $a(y,z,\sigma)$ or a derivative of $Q_g(y,\sigma z)$; by the bounds above, such terms are of order at least $\sigma$. Therefore, 
\[
|\partial_y^\alpha G(y,z,\sigma)|
\le
C_\alpha \sigma (1+\|z\|^{m_\alpha})e^{-(3c_\varepsilon/8)\|z\|^2}.
\]
Since the right-hand sides are integrable on $\mathbb R^d$ and $\eta_\sigma(z)\le 1$, by the dominated convergence it follows that
\[
|E_1(y,\sigma)|\le C(\sigma\|v_y\|+\sigma^2),
\qquad
|\partial_y^\alpha E_1(y,\sigma)|\le C_\alpha \sigma,\quad |\alpha|\le 3.
\]
Combining the bounds for $E_1$ and $E_2$, we conclude that
\[
I_1(y,\sigma)
=
e^{-\|v_y\|^2/(2\sigma^2)}\sigma^d
\big(
\frac{(2\pi)^{d/2}}{\sqrt{\det A_y}}+E(y,\sigma)
\big),
\]
where
\[
|E(y,\sigma)|\le C(\sigma\|v_y\|+\sigma^2),
\qquad
|\partial_y^\alpha E(y,\sigma)|\le C_\alpha \sigma,\quad |\alpha|\le 3.
\]

%I_2
We finally estimate the complement term
\[
I_2(y,\sigma)
=
\int_M (1-\chi_y(x))
\exp\big(-\frac{\|y-x\|^2}{2\sigma^2}\big)d\mu(x).
\]
By construction of the cutoff, $(1-\chi_y(x))\neq 0$ implies $d_M(x,\pi(y))\ge r/2$.
Consider the compact set
\[
K:=\{(y,x): y\in T(\tau-\varepsilon),\ d_M(x,\pi(y))\ge r/2\}.
\]
Since $\pi(y)$ is the unique nearest point on $M$ to $y$, the continuous function
\[
f(y,x):=\|y-x\|^2-\|v_y\|^2
\]
is strictly positive on $K$. Hence, by compactness, there exists $\delta_{\varepsilon,r}>0$ such that
\[
\|y-x\|^2\ge \|v_y\|^2+\delta_{\varepsilon,r}
\]
whenever $(1-\chi_y(x))\neq 0$.
Therefore
\[
|I_2(y,\sigma)|
\le
\int_M (1-\chi_y(x))
\exp\big(-\frac{\|y-x\|^2}{2\sigma^2}\big)d\mu(x)
\le
V_M
\exp\big(-\frac{\|v_y\|^2+\delta_{\varepsilon,r}}{2\sigma^2}\big).
\]
To estimate derivatives, note that
\[
\Big|\partial_y^\alpha\Big(
(1-\chi_y(x))
\exp\big(-\frac{\|y-x\|^2}{2\sigma^2}\big)
\Big)\Big|
\le
C_\alpha \sigma^{-m_\alpha}
\exp\big(-\frac{\|v_y\|^2+\delta_{\varepsilon,r}}{2\sigma^2}\big).
\]
Since the above bound is integrable, we may differentiate under the integral sign. Then by Leibniz' rule and $\|y-x\|^2\ge \|v_y\|^2+\delta_{\varepsilon,r}$, 
\[
|\partial_y^\alpha I_2(y,\sigma)|
\le
C_\alpha \sigma^{-m_\alpha}
\exp\big(-\frac{\|v_y\|^2+\delta_{\varepsilon,r}}{2\sigma^2}\big).
\]
Thus $I_2$ and all its derivatives up to order three are exponentially negligible compared with the
main expansion, and may be absorbed into the remainder term.

%conclusion
Finally, putting $I_1$ and $I_2$ together. Since $A_y$ is uniformly positive definite, the factor
$(\det A_y)^{-1/2}$ is bounded above and below. Hence the error coming from $E(y,\sigma)$ and $I_2(y,\sigma)$ may be normalized by this principal term. Hence, we obtain
\[
P_\sigma(y)
=
\frac{1}{V_M(2\pi\sigma^2)^{\frac{D-d}{2}}}
e^{-\|v_y\|^2/(2\sigma^2)}
\frac{1}{\sqrt{\det A_y}}
\big(
1 + R_P(y,\sigma)
\big),
\]
where 
\[
|R_P(y,\sigma)|\le C(\sigma\|v_y\|+\sigma^2),
\qquad
|\partial_y^\alpha R_P(y,\sigma)|\le C_\alpha \sigma,\quad |\alpha|\le 3.
\]
This refined factorization also provides the derivative bounds needed in Corollary~\ref{co:log-density}.
\end{proof}

\subsubsection{Proof of Corollary~\ref{co:log-density}}\label{sec:pf-log density}
\begin{proof}
Following the proof of Theorem~\ref{th:probability-noisy}, taking logarithms in the expansion of $P_\sigma(y)$ yields
\[
\log P_\sigma(y)
=
\log \big(V_M(2\pi\sigma^2)^{\frac{D-d}{2}}\big)^{-1}
-\frac{\|v_y\|^2}{2\sigma^2}
-\frac12\log\det A_y
+\log(1+R_P(y,\sigma)),
\]
where 
\[
|R_P(y,\sigma)|\le C(\sigma\|v_y\|+\sigma^2),
\qquad
|\partial_y^\alpha R_P(y,\sigma)|\le C_\alpha \sigma,\quad |\alpha|\le 3.
\]
Set
\[
R(y,\sigma):=\log(1+R_P(y,\sigma)).
\]
Since $1+R_P$ stays away from zero after decreasing $\sigma_0$ if necessary, the same bound carries over to $R(y,\sigma)$ with
\[
R(y,\sigma)=O(\sigma\|v_y\|+\sigma^2).
\]
Furthermore, by repeated differentiation of $\log(1+R_P)$, together with the above bounds on $\partial_y^\alpha R_P$, it follows that
\[
|\partial_y^\alpha R(y,\sigma)| \le C_{\alpha}\sigma,
\qquad |\alpha|\le3,
\]
which satisfies the stated bounds.
\end{proof}

\subsubsection{Proof of Theorem~\ref{th:gradient-noisy}}\label{sec:pf-gradient noisy} 
\begin{proof}
Corollary~\ref{co:log-density} gives the expansion
\[
\log P_\sigma(y)
= \log \big(V_M(2\pi\sigma^2)^{\frac{D-d}{2}}\big)^{-1} - \frac{\|v_y\|^2}{2\sigma^2}
  - \frac{1}{2}\log\det A_y + R(y,\sigma),
\]
where the remainder $R$ satisfies,
\[
\|\nabla_y^m R(y,\sigma)\|_{\mathrm{op}}
\le C_{m}\sigma, \qquad m\le 3.
\]
Taking the ambient gradient yields
\begin{equation}\label{eq:gradient-decomposition}
\mathcal G_\sigma(y) = \nabla \log P_\sigma(y)
= - \frac{1}{2\sigma^2}\nabla\|v_y\|^2
  - \frac{1}{2}\nabla \log\det A_y + \nabla R(y,\sigma).
\end{equation}
%\|v_y\|^2
Using Lemma~\ref{lem:derivative-of-v-pi}, for any $\omega \in \mathbb{R}^D$ we have
\[
\nabla_\omega v_y = \omega - A_y^{-1}\omega^\top.
\]
Since $v_y \in T_{\pi(y)}^\perp M$ and $A_y^{-1}\omega^\top \in T_{\pi(y)}M$, we obtain
\[
D_\omega \|v_y\|^2
= 2\langle v_y, \nabla_\omega v_y \rangle
= 2\langle v_y, \omega \rangle.
\]
By the definition of the gradient, this implies
\[
\nabla \|v_y\|^2 = 2 v_y.
\]
Consequently, the first term in \eqref{eq:gradient-decomposition} is exactly
\[
- \frac{1}{2\sigma^2}\nabla\|v_y\|^2 = - \frac{v_y}{\sigma^2}.
\]
%log det A_y
Define
\[
F(y) := -\frac{1}{2}\log\det A_y, \qquad y \in T(\tau-\varepsilon).
\]
For $x\in M$ we have $A_x=I_{T_xM}$, hence $F|_M\equiv 0$. Therefore all tangential derivatives of $F$ vanish on $M$. Then it suffices to compute the normal derivatives. Let $u \in T_x^\perp M$ with $\|u\| = 1$ and consider the curve
\[
\gamma(t) := x + t u, \qquad |t| \ll 1.
\]
For $|t|$ small, $\gamma(t) \in T(\tau-\varepsilon)$, $\pi(\gamma(t)) = x$ and $v_{\gamma(t)} = t u$. Let $S_u$ be the shape operator defined at $x$ associated with $u$. By linearity of the shape operator in the normal argument, we have
\[
A_{\gamma(t)} = I_{T_xM} - \langle t u, \Pi_x \rangle
= I_{T_xM} - t S_u,
\]
Therefore
\[
F(\gamma(t)) = -\frac{1}{2}\log\det\big(I_{T_xM} - t S_u\big).
\]
Differentiating at $t=0$ gives
\[
\frac{d}{dt} F(\gamma(t))\Big|_{t=0}
= -\frac{1}{2}\operatorname{Tr}(-S_u)
= \frac{1}{2}\operatorname{Tr}(S_u).
\]
For mean curvature $H_x = \frac{1}{d}\operatorname{Tr}(\Pi_x)$, we have $\operatorname{Tr}(S_u) = d \langle H_x, u\rangle$. Hence
\[
\frac{d}{dt} F(\gamma(t))\Big|_{t=0}
= \frac{d}{2} \langle H_x, u\rangle,
\]
which shows that
\[
\nabla F(x) = \frac{d}{2} H_x \in T_x^\perp M.
\]
Now let $y \in T(\tau-\varepsilon)$. By the regularity of $A_y$, $F$ is $C^3$ on $T(\tau-\varepsilon)$ and $\nabla F$ is Lipschitz on $T(\tau-\varepsilon)$. Therefore
\[
\|\nabla F(y) - \nabla F(\pi(y))\| \le C\|y-\pi(y)\| =  C \|v_y\|.
\]
Then
\[
-\frac{1}{2}\nabla \log\det A_y
= \nabla F(y)
= \frac{d}{2} H_{\pi(y)} + O(\|v_y\|).
\]
%remainder
For $\nabla R$, by Corollary~\ref{co:log-density} with $m=1$,
\[
\|\nabla R(y,\sigma)\|
\le C_{1}\sigma.
\]
Substituting the above gradients into \eqref{eq:gradient-decomposition}, we obtain
\[
\mathcal G_\sigma(y)
= - \frac{v_y}{\sigma^2}
  + \frac{d}{2}H_{\pi(y)}
  + O(\|v_y\| + \sigma).
\]
\end{proof}

\subsubsection{Proof of Theorem~\ref{th:noisy-hessian}}\label{sec:pf-noisy hessian}
\begin{proof}
By Corollary~\ref{co:log-density},
\[
\log P_\sigma(y)
= \log \big(V_M(2\pi\sigma^2)^{\frac{D-d}{2}}\big)^{-1} - \frac{\|v_y\|^2}{2\sigma^2}
  - \frac{1}{2}\log\det A_y + R(y,\sigma).
\]
Taking two ambient derivatives in $y$ we obtain
\[
\mathcal{H}_\sigma(y)
:= \nabla^2 \log P_\sigma(y)
= - \frac{1}{2\sigma^2} \nabla^2 \|v_y\|^2
  - \frac{1}{2}\nabla^2 \log\det A_y
  + \nabla^2 R(y,\sigma).
\]
We treat the three terms separately.
% \|v_y\|^2
As in the proof of Theorem~\ref{th:gradient-noisy}, we have
\[
\nabla \|v_y\|^2 = 2 v_y,
\]
For any $\omega \in \mathbb{R}^D$, Lemma~\ref{lem:derivative-of-v-pi} yields
\[
\nabla_\omega v_y = \omega - A_y^{-1}\omega^\top.
\]
Then, the Jacobian of $v_y$ is
\[
\nabla v_y = I_D - A_y^{-1} P_T(y),
\]
Therefore
\[
- \frac{1}{2\sigma^2}\nabla^2 \|v_y\|^2 = -\frac{1}{\sigma^2} \big( I_D- A_y^{-1}P_T(y) \big) = -\frac{1}{\sigma^2}P_N(y) - \frac{1}{\sigma^2}(I_{T_{\pi(y)}M}-A_y^{-1})P_T(y)
\]
%log det A_y
Let $F(y) := -\frac{1}{2}\log\det A_y$. Since $F(y)$ is $C^3$ on $T(\tau-\varepsilon)$, its Hessian is bounded:
\[
\|\nabla^2 F(y)\|_{\mathrm{op}}=O(1).
\]
For $\nabla^2 R$, by Corollary~\ref{co:log-density} with $m=2$,
\[
\|\nabla^2 R(y,\sigma)\|_{\mathrm{op}}
\le C_{\varepsilon,2} \sigma.
\]
Combining the above three terms, we can write
\[
\mathcal{H}_\sigma(y)
= - \frac{1}{\sigma^2}P_N(y)
  - \frac{1}{\sigma^2}(I_{T_{\pi(y)}M} - A_y^{-1})P_T(y)
  + O(1).
\]
\end{proof}

\subsubsection{Proof of Corollary~\ref{co:noisy-hessian-gap}}\label{sec:pf-noisy hessian gap}
\begin{proof}
By Theorem~\ref{th:noisy-hessian}, we may write
\[
\mathcal H_\sigma(y)= \mathcal H_0(y)+E(y,\sigma),
\]
where
\[
\mathcal H_0(y):=
-\frac{1}{\sigma^2}P_N(y)
-\frac{1}{\sigma^2}(I_{T_{\pi(y)}M}-A_y^{-1})P_T(y)
\]
and 
\[
\|E(y,\sigma)\|_{\mathrm{op}}\le C.
\]
Then, $\mathcal H_0(y)$ is block diagonal:
\[
\mathcal H_0(y)\big|_{T_{\pi(y)}^\perp M}=-\sigma^{-2}I,
\qquad
\mathcal H_0(y)\big|_{T_{\pi(y)}M}=-\sigma^{-2}(I_{T_{\pi(y)}M}-A_y^{-1}).
\]
Hence the $D-d$ normal eigenvalues of $\mathcal H_0(y)$ are all equal to $-\sigma^{-2}$, while its $d$ tangential eigenvalues are those of $-\sigma^{-2}(I_{T_{\pi(y)}M}-A_y^{-1})$.
Now recall that
\[
A_y=I_{T_{\pi(y)}M}-\langle v_y,\Pi_{\pi(y)}\rangle.
\]
By $y\in T(\tau-\varepsilon)$ and the reach bound,
\[
\|\langle v_y,\Pi_{\pi(y)}\rangle\|_{\mathrm{op}}
\le \frac{\|v_y\|}{\tau}
\le \frac{\tau-\varepsilon}{\tau} = 1-\frac{\varepsilon}{\tau}
<1.
\]
Therefore $A_y$ is uniformly invertible on $T(\tau-\varepsilon)$. Moreover, every eigenvalue $\lambda$ of $A_y$ satisfies
\[
\lambda\le 1+\|\langle v_y,\Pi_{\pi(y)}\rangle\|_{\mathrm{op}}
\le 2-\frac{\varepsilon}{\tau},
\]
so every eigenvalue of $A_y^{-1}$ is bounded below by
\[
\lambda_{\min}(A_y^{-1})\ge \frac{1}{2-\varepsilon/\tau}.
\]
If $\mu$ is a tangential eigenvalue of $\mathcal H_0(y)$, then
\[
\mu=-\sigma^{-2}(1-\lambda^{-1}),
\]
and thus
\[
\big|\mu+\sigma^{-2}\big|
=\sigma^{-2}\lambda^{-1}
\ge \frac{1}{2-\varepsilon/\tau}\sigma^{-2}.
\]
It follows that the tangential and normal spectral clusters of $\mathcal H_0(y)$ are separated by at least
\[
c_{\varepsilon,0}\sigma^{-2},
\qquad
c_{\varepsilon,0}:=\frac{1}{2-\varepsilon/\tau}.
\]
In particular, every tangential eigenvalue is strictly larger than every normal eigenvalue, so the tangential cluster coincides with the $d$ largest eigenvalues of $\mathcal H_0(y)$.

Finally, each eigenvalue of $\mathcal H_\sigma(y)$ differs from the corresponding eigenvalue of $\mathcal H_0(y)$ by at most $\|E(y, \sigma)\|_{\mathrm{op}}\le C$. Hence the gap between the $d$ largest eigenvalues of $\mathcal H_\sigma(y)$ and the remaining $D-d$ eigenvalues is bounded below by
\[
c_{\varepsilon,0}\sigma^{-2}-2C.
\]
After decreasing $\sigma_0$ if necessary, this lower bound is at least
\[
c_\varepsilon\sigma^{-2}
\]
for some constant $c_\varepsilon>0$ depending only on $\varepsilon$ and $M$. This proves the claim.
\end{proof}

\subsection{Proofs for Section~\ref{sec:compute-geometry}}\label{sec:pf-geometry}
We next prove the results in Section~\ref{sec:compute-geometry}.
\subsubsection{Proof of Theorem~\ref{th:tangent-space}}\label{sec:pf-tangent-space}
\begin{proof}
Write $x:=\pi(y)$. Recall that
\[
\mathcal H_0(y):=
-\frac1{\sigma^2}P_N(y)
-\frac1{\sigma^2}(I_{T_xM}-A_y^{-1})P_T(y),
\]
where
\[
\mathcal H_\sigma(y)=\mathcal H_0(y)+R_\sigma(y),
\qquad
\|R_\sigma(y)\|_{\mathrm{op}}\le C.
\]

By Corollary~\ref{co:noisy-hessian-gap}, both $\mathcal H_0(y)$ and $\mathcal H_\sigma(y)$ have a spectral gap of size at least $c_\varepsilon \sigma^{-2}$ between the largest $d$ eigenvalues and the remaining $D-d$ eigenvalues.
The operator $\mathcal H_0(y)$ is block diagonal with respect to $\mathbb R^D=T_xM\oplus T_x^\perp M$. Its spectral projector associated with the $d$ largest eigenvalues is therefore exactly $P_T(y)$.
Thus, applying the Davis--Kahan theorem gives
\[
\sin\{\Theta(\widehat T_yM, T_xM)\} 
\le
C\frac{\|R_\sigma(y)\|_{\mathrm{op}}}{c_\varepsilon \sigma^{-2}}
\le C' \sigma^2.
\]

Next consider the consistency of $\widehat d$. Let $\eta_1\ge\cdots\ge \eta_d>\eta_{d+1}=\cdots = \eta_D$ be the eigenvalues of $\mathcal H_0(y)$, and $\lambda_1\ge\cdots\ge \lambda_d>\lambda_{d+1}=\cdots = \lambda_D$ be the eigenvalues of $\mathcal H_\sigma(y)$.
Since $A_y=I_{T_xM}-\langle v_y,\Pi_{x}\rangle$ and $\|\Pi_{x}\|_{\mathrm{op}}\le \tau^{-1}$, we have $\|I_{T_xM} - A_y^{-1}\|_{\mathrm{op}}=O(\|v_y\|)$ uniformly on $T(\tau-\varepsilon)$. Thus
\[
\sup_{1\le i\le d}|\eta_i|
\le C\frac{\|v_y\|}{\sigma^2}.
\]
Therefore,
\[
\sup_{k\neq d} |\eta_k-\eta_{k+1}|
\le C\frac{\|v_y\|}{\sigma^2},
\qquad
|\eta_d-\eta_{d+1}|
\ge c_\varepsilon \sigma^{-2}.
\]
By Weyl's inequality,
\[
|\lambda_k-\eta_k|\le C,
\qquad 1\le k\le D,
\]
so
\[
\sup_{k\neq d} |\lambda_k-\lambda_{k+1}|
\le C_1\frac{\|v_y\|}{\sigma^2}+C_2,
\]
whereas
\[
|\lambda_d-\lambda_{d+1}|
\ge c_\varepsilon \sigma^{-2}-C_3.
\]
Hence, along any sequence such that $\sigma\to0$ and $\|v_y\|\to0$, the gap at $k=d$ is eventually strictly larger than all other consecutive gaps, i.e.,
\[
\widehat d
= \arg\max_{k}|\lambda_k-\lambda_{k+1}| = d.
\]
\end{proof}

\subsubsection{Proof of Theorem~\ref{th:umbilical-submanifold}}\label{sec:pf-umbilical submanifold}
\begin{proof}
For $x\in M$, the estimator in the theorem is
\[
\widehat\Pi_x(u,v)
=
\frac{2}{d}g(u,v)\mathcal G_\sigma(x),
\]
while the true second fundamental form is
\[
\Pi_x(u,v) = g(u,v)H_x,
\]
because $M$ is totally umbilical. By Theorem~\ref{th:gradient-noisy}, for points $x\in M$ we have
\[
\mathcal G_\sigma(x)
=
\frac{d}{2}H_x + R_\sigma(x),
\qquad
\|R_\sigma(x)\| \le C\sigma.
\]
Therefore
\[
\widehat\Pi_x(u,v)-\Pi_x(u,v)
=
g(u,v)\big(\frac{2}{d}\mathcal G_\sigma(x)-H_x\big)
=
g(u,v)\frac{2}{d}R_\sigma(x).
\]
Taking the operator norm over the unit tangent vectors yields
\[
\|\widehat\Pi_x-\Pi_x\|_{\mathrm{op}} 
\le
\frac{2}{d}\|R_\sigma(x)\|
\le C\sigma.
\]
\end{proof}

\subsubsection{Proof of Theorem~\ref{th:hypersurface}}\label{sec:pf-hypersurface}
\begin{proof}
Write $x:=\pi(y)$ and $r:=\|v_y\|$. We define two vector fields on $U_\sigma$ by
\[
\mathcal N(y):=P_N(y)\mathcal G_\sigma(y),
\qquad
\mathcal T(y):=P_T(y)\mathcal G_\sigma(y),
\]
so that
\[
\mathcal G_\sigma(y)=\mathcal N(y)+\mathcal T(y).
\]
Recall from the proof of Theorem~\ref{th:gradient-noisy} that \[
\mathcal G_\sigma(y)
=
-\frac{v_y}{\sigma^2}+\nabla F(y)+\nabla R(y,\sigma),
\qquad
F(y):=-\frac12\log\det A_y,
\]
where 
\[
\nabla F(y)=\frac d2 H_{\pi(y)}+O(\|v_y\|), \quad \|\nabla R(y,\sigma)\|=O(\sigma).
\]
Since $v_y, H_{\pi(y)}\in T_{\pi(y)}^\perp M$, we know
\begin{equation}\label{eq:T-size}
\|\mathcal T(y)\|=O(r+\sigma),
\end{equation}
and there exists $C>0$ such that
\begin{equation}\label{eq:N-size}
\|\mathcal N(y)\|\ge \frac{r}{\sigma^2} -C \qquad
\text{for all }y\in U_\sigma,\ 0<\sigma\le\sigma_0.
\end{equation}
We show that $\|\mathcal N(y)\|$ is lower bounded on $U_\sigma$. Let $C_0 \ge C + \frac{c_0}{2}$.
Since $\|\mathcal G_\sigma(y)\|\ge c_0$ uniformly on $U_\sigma$, when $r\le C_0\sigma^2$ we have $\|\mathcal N(y)\| \ge \|\mathcal G_\sigma(y)\| - \|\mathcal T(y)\| \ge c_0/2$ for sufficiently small $\sigma_0$; For $r>C_0\sigma^2$, $\|\mathcal N(y)\|\ge c_0/2$ by \eqref{eq:N-size}. Therefore we have $\|\mathcal N(y)\|\ge c_0/2$ uniformly on $U_\sigma$. 

For any $C^1$ vector field $V$ on $U_\sigma$ with $V(y)\neq 0$, define
\[
Q_y(V)(u,v)
:=
-\frac{\langle \nabla_{A_yu}V(y),v\rangle}{\|V(y)\|^2}V(y).
\]
Since $M$ is a hypersurface and $\mathcal N(y)\in T_x^\perp M\setminus\{0\}$, Lemma~\ref{lem:extended-normal-vector-field} yields
\begin{equation}\label{eq:Pi-as-QN}
\Pi_x=Q_y(\mathcal N).
\end{equation}
We will prove
\begin{equation}\label{eq:hypersurface-split}
\|\widehat\Pi_y-Q_y(\mathcal G_\sigma)\|_{\mathrm{op}}=O(r+\sigma),
\qquad
\|Q_y(\mathcal G_\sigma)-Q_y(\mathcal N)\|_{\mathrm{op}}=O(r+\sigma).
\end{equation}
Together with \eqref{eq:Pi-as-QN}, this implies the theorem.

We first compare $Q_y(\mathcal G_\sigma)$ and $Q_y(\mathcal N)$. Consider the derivative of $\mathcal T$ along tangent directions. Write
\[
W(y):=P_T(y)\nabla F(y),
\qquad
S(y):=P_T(y)\nabla R(y,\sigma),
\]
so that $\mathcal T=W+S$.
We first treat $W$. For every $x'\in M$, $\nabla F(x')=\frac d2 H_{x'}\in T_{x'}^\perp M$, hence
\[
W(x')=P_T(x')\nabla F(x')=0,
\qquad x'\in M.
\]
Thus $W|_M\equiv 0$, which implies
\[
\langle \nabla_u W(x),v\rangle=0.
\]
Using
\[
\langle \nabla_{A_yu}W(y),v\rangle
=
\langle (\nabla W(y)-\nabla W(x))A_yu,v\rangle
+
\langle \nabla W(x)(A_yu-u),v\rangle
+
\langle \nabla_uW(x),v\rangle,
\]
then the last term vanishes. Since $W$ is $C^2$ on the compact set under consideration, $\nabla W$ is Lipschitz, and thus
\[
\|\nabla W(y)-\nabla W(x)\|_{\mathrm{op}}=O(r),
\]
and $A_y$ is uniformly bounded on $T_xM$, the first term is $O(r)$.
Moreover, since $A_yu-u=-(\nabla_u v_y)^\top$ and $\Pi$ is uniformly bounded on $M$,
\[
\|A_yu-u\|=O(r)\|u\|.
\]
As $\nabla W$ is uniformly bounded on $U_\sigma$, the second term is also
$O(r)$. Hence
\begin{equation}\label{eq:W-derivative}
\big|\langle \nabla_{A_yu}W(y),v\rangle\big|
=O(r).
\end{equation}
Since $S=P_T\nabla R$, the product rule gives
\[
\nabla S=(\nabla P_T)\nabla R + P_T\nabla^2R.
\]
Because $P_T$ and $\nabla P_T$ are uniformly bounded on $U_\sigma$, while by Corollary~\ref{co:log-density} $\|\nabla R(y,\sigma)\|=O(\sigma)$ and $\|\nabla^2R(y,\sigma)\|_{\mathrm{op}} = O(\sigma)$, we have  
\begin{equation}\label{eq:S-derivative}
\big|\langle \nabla_{A_yu}S(y),v\rangle\big|
=O(\sigma).
\end{equation}
Combining \eqref{eq:W-derivative} and \eqref{eq:S-derivative}, we obtain
\begin{equation}\label{eq:T-derivative}
\big|\langle \nabla_{A_yu}\mathcal T(y),v\rangle\big|
=O(r+\sigma).
\end{equation}
Set
\[
\beta_{\mathcal N}(u,v):=\langle \nabla_{A_yu}\mathcal N(y),v\rangle,
\qquad
\beta_{\mathcal T}(u,v):=\langle \nabla_{A_yu}\mathcal T(y),v\rangle.
\]
Then
\[
Q_y(\mathcal G_\sigma)(u,v)-Q_y(\mathcal N)(u,v)
=
-\frac{\beta_{\mathcal T}(u,v)}{\|\mathcal G_\sigma(y)\|^2}\mathcal G_\sigma(y)
-\beta_{\mathcal N}(u,v)
\big(
\frac{\mathcal G_\sigma(y)}{\|\mathcal G_\sigma(y)\|^2}
-
\frac{\mathcal N(y)}{\|\mathcal N(y)\|^2}
\big).
\]
By \eqref{eq:T-derivative} and the assumption
$\|\mathcal G_\sigma(y)\|\ge c_0$ on $U_\sigma$,
\[
\Big\|
\frac{\beta_{\mathcal T}(u,v)}{\|\mathcal G_\sigma(y)\|^2}\mathcal G_\sigma(y)
\Big\|
\le
\frac{|\beta_{\mathcal T}(u,v)|}{\|\mathcal G_\sigma(y)\|}
=
O(r+\sigma).
\]
For the second term, note that \eqref{eq:Pi-as-QN} implies
\[
Q_y(\mathcal N)(u,v)
=
-\frac{\beta_{\mathcal N}(u,v)}{\|\mathcal N(y)\|^2}\mathcal N(y)
=
\Pi_x(u,v),
\]
hence
\[
\frac{|\beta_{\mathcal N}(u,v)|}{\|\mathcal N(y)\|}
\le
\|\Pi_x\|_{\mathrm{op}}.
\]
Since $\|\mathcal{G_\sigma}\|\ge c_0$ and $\|\mathcal N(y)\|\ge c_0/2$ uniformly on $U_\sigma$, the map
\[
a\longmapsto \frac{a}{\|a\|^2}
\]
is uniformly Lipschitz on the set
$\{a\in\mathbb R^D:\|a\|\ge c_0/2\}$. Thus
\[
\Big\|
\frac{\mathcal G_\sigma(y)}{\|\mathcal G_\sigma(y)\|^2}
-
\frac{\mathcal N(y)}{\|\mathcal N(y)\|^2}
\Big\|
=
O(\|\mathcal T(y)\|)
=
O(r+\sigma).
\]
Hence the second term is also $O(r+\sigma)$. We conclude that
\begin{equation}\label{eq:QG-QN}
\|Q_y(\mathcal G_\sigma)-Q_y(\mathcal N)\|_{\mathrm{op}}
=
O(r+\sigma).
\end{equation}
Next, we compare $\widehat\Pi_y$ with $Q_y(\mathcal G_\sigma)$. We have
\[
\widehat\Pi_y(u,v) - Q_y(\mathcal G_\sigma)(u,v)
=
-\frac{\langle \mathcal H_\sigma(y)(u-A_yu),v\rangle}{\|\mathcal G_\sigma(y)\|^2}
\mathcal G_\sigma(y).
\]
Hence
\begin{equation}\label{eq:hatPi-minus-tildePi}
\|\widehat\Pi_y-Q_y(\mathcal G_\sigma)\|_{\mathrm{op}}
\le
\frac{
|\langle \mathcal H_\sigma(y)(u-A_yu),v\rangle|
}{
\|\mathcal G_\sigma(y)\|
}.
\end{equation}
Let $w:=u-A_yu$. Then $w\in T_xM$ and, as above, $\|w\|=O(r)\|u\|$.
By Theorem~\ref{th:noisy-hessian}, we can write
\[
\mathcal H_\sigma(y)
=
-\frac{1}{\sigma^2}P_N(y)
+
O\big(\frac{r}{\sigma^2}\big)P_T(y)
+
O(1).
\]
Thus
\[
|\langle \mathcal H_\sigma(y)w,v\rangle|
\le
C\frac{r}{\sigma^2}\|w\|
+
C\|w\|.
\]
Using $\|w\|=O(r)\|u\|$, we obtain
\begin{equation}\label{eq:numerator-bound}
|\langle \mathcal H_\sigma(y)(u-A_yu),v\rangle|
\le
C\big(\frac{r^2}{\sigma^2}+r\big).
\end{equation}
By Theorem~\ref{th:gradient-noisy}, there exists $C_1>0$ such that
\begin{equation}\label{eq:G-lower}
\|\mathcal G_\sigma(y)\|
\ge
\frac{r}{\sigma^2}-C_1
\qquad
\text{for all }y\in U_\sigma,\ 0<\sigma\le\sigma_0.
\end{equation}
We distinguish two cases. If $r\le 2C_1\sigma^2$, then $r=O(\sigma^2)$, and by
\eqref{eq:numerator-bound} together with the assumption
$\|\mathcal G_\sigma(y)\|\ge c_0$ on $U_\sigma$,
\[
\|\widehat\Pi_y-Q_y(\mathcal G_\sigma)\|_{\mathrm{op}}
\le
C\big(\frac{r^2}{\sigma^2}+r\big)
=
O(r)
=
O(r+\sigma).
\]
If $r>2C_1\sigma^2$, then \eqref{eq:G-lower} yields
\[
\|\mathcal G_\sigma(y)\|\ge \frac12\frac{r}{\sigma^2},
\]
and therefore, by \eqref{eq:hatPi-minus-tildePi} and \eqref{eq:numerator-bound},
\[
\|\widehat\Pi_y-Q_y(\mathcal G_\sigma)\|_{\mathrm{op}}
\le
C
\frac{r^2/\sigma^2+r}{r/\sigma^2}
=
O(r+\sigma^2)
=
O(r+\sigma).
\]
In either case,
\begin{equation}\label{eq:hatPi-tildePi-final}
\|\widehat\Pi_y-Q_y(\mathcal G_\sigma)\|_{\mathrm{op}}=O(r+\sigma).
\end{equation}
Combining \eqref{eq:Pi-as-QN}, \eqref{eq:QG-QN}, and \eqref{eq:hatPi-tildePi-final} yields
\[
\|\widehat\Pi_y-\Pi_x\|_{\mathrm{op}}
\le
\|\widehat\Pi_y-Q_y(\mathcal G_\sigma)\|_{\mathrm{op}}
+
\|Q_y(\mathcal G_\sigma)-Q_y(\mathcal N)\|_{\mathrm{op}} = O(r+\sigma)
=
O(\|v_y\|+\sigma).
\] 
This proves the claim.
\end{proof}

\subsubsection{Proof of Corollary~\ref{co:level-set}}\label{sec:pf-level set} 
\begin{proof} 
Let $L_{\sigma,c}$ be a level set of $P_\sigma$ such that $L_{\sigma,c}\cap M\neq\emptyset$, and choose $x\in L_{\sigma,c}\cap M$.

We first show that $L_{\sigma,c}\subset T(\tau-\varepsilon)$ for all sufficiently small $\sigma$. By Theorem~\ref{th:probability-noisy}, uniformly for $x\in M$,
\[
P_\sigma(x)
=
\frac{1}{V_M(2\pi\sigma^2)^{\frac{D-d}{2}}}\big(1+O(\sigma^2)\big).
\]
On the other hand, if $y\notin T(\tau-\varepsilon)$, then $\|y-z\|\ge \tau-\varepsilon$ for every $z\in M$, so
\[
P_\sigma(y)
=
\frac{1}{V_M(2\pi\sigma^2)^{D/2}}
\int_M \exp\big(-\frac{\|y-z\|^2}{2\sigma^2}\big)d\mu(z)
\le
\frac{1}{(2\pi\sigma^2)^{D/2}}
\exp\big(-\frac{(\tau-\varepsilon)^2}{2\sigma^2}\big).
\]
Hence
\[
\sup_{y\notin T(\tau-\varepsilon)}\frac{P_\sigma(y)}{P_\sigma(x)}
\le
C \sigma^{-d}
\exp\big(-\frac{(\tau-\varepsilon)^2}{2\sigma^2}\big)\to 0.
\]
Therefore $P_\sigma(y)<P_\sigma(x)$ for all such $y$ when $\sigma$ is small enough, which implies every $y\in L_{\sigma,c}$ belongs to $T(\tau-\varepsilon)$ and $L_{\sigma,c}\subset T(\tau-\varepsilon)$.

Now fix $y\in L_{\sigma,c}$ and write $r:=d(y,M)=\|v_y\|$. Since $P_\sigma(y)=P_\sigma(x)$, Corollary~\ref{co:log-density} gives
\[
0=\log P_\sigma(y)-\log P_\sigma(x)
=-\frac{r^2}{2\sigma^2}-\frac12\log\det A_y+R(y,\sigma)-R(x,\sigma),
\]
where
\[
A_y=I_{T_{\pi(y)}M}-\langle v_y,\Pi_{\pi(y)}\rangle,
\qquad
|R(y,\sigma)|\le C(\sigma r+\sigma^2),
\qquad
|R(x,\sigma)|\le C\sigma^2.
\]

Moreover, by Lemma~\ref{lem:second-fundamental-form-bound},
\[
\|\langle v_y,\Pi_{\pi(y)}\rangle\|_{\mathrm{op}}
\le \|v_y\|\|\Pi_{\pi(y)}\|_{\mathrm{op}}
\le r/\tau
\le (\tau-\varepsilon)/\tau<1.
\]
Thus $A_y$ is uniformly positive definite on $T(\tau-\varepsilon)$, and since $\log\det$ is smooth on this set,
\[
|\log\det A_y|\le C\|A_y-I_{T_{\pi(y)}M}\|_{\mathrm{op}}\le Cr.
\]
Therefore
\[
\frac{r^2}{2\sigma^2}
\le C r + C\sigma r + C\sigma^2
\le C(r+\sigma^2)
\]
for all sufficiently small $\sigma$. Multiplying by $\sigma^2$ yields
\[
r^2\le C\sigma^2 r + C\sigma^4.
\]
This implies
\[
r \le C\sigma^2.
\]
Since all constants above are uniform over $y\in L_{\sigma,c}$ and over all level sets with $L_{\sigma,c}\cap M\neq\emptyset$, taking the supremum over $y\in L_{\sigma,c}$ proves the claim.
\end{proof}

\subsubsection{Proof of Theorem~\ref{th:general-manifold}}\label{sec:pf-general manifold}
\begin{proof}
Write $x:=\pi(y)$. Recall that $P_T(y)$ denotes the orthogonal projection onto $T_xM$, while $\widehat P_T(y)$ denotes the spectral projector onto the $d$ largest eigenvalues of $\mathcal{H}_\sigma(y)$. By Theorem~\ref{th:tangent-space},
\[
\|\widehat P_T(y)-P_T(y)\|_{\mathrm{op}}
\le C\sigma^2.
\] 
By Lemma~\ref{lem:second-fundamental-form-general} applied at $x$ we have,
\[
\Pi_x(u,v) = \nabla_uP_T(x)v.
\]
The estimator in the theorem is
\[
\widehat\Pi_y(u,v)
=
\nabla_u\widehat P_T(y)v.
\]
We decompose
\[
\widehat\Pi_y(u,v)-\Pi_x(u,v)
=
\big(\nabla_u\widehat P_T(y)-\nabla_uP_T(y)\big)v
+
\big(\nabla_uP_T(y)-\nabla_uP_T(x)\big)v.
\]
%Lip bound
For the second term, since $P_T(y)$ is $C^{k-1}$ on $T(\tau-\varepsilon)$, $\nabla P_T(y)$ is Lipschitz on $T(\tau-\varepsilon)$, i.e.,
\[
\|\nabla_uP_T(y)-\nabla_uP_T(x)\|_{\mathrm{op}}
\le C\|v_y\|.
\]
%asymptotic error
Next we consider the first term. Let
\[
P_N(y):=I_D-P_T(y),
\qquad
\mathcal H_0(y):= -\sigma^{-2}P_N(y)-\sigma^{-2}(I_{T_xM}-A_y^{-1})P_T(y).
\]
Then $P_T(y)$ is exactly the spectral projector of $\mathcal H_0(y)$ associated with the largest $d$ eigenvalues.

Let
\[
\widetilde {\mathcal H}_0(y):=\sigma^2 \mathcal H_0(y),
\qquad
\widetilde {\mathcal H}_\sigma(y):=\sigma^2 \mathcal H_\sigma(y).
\]
By the uniform boundedness of $A_y^{-1}$ on $T(\tau-\varepsilon)$ and the expansion of $\mathcal H_\sigma$ in Theorem~\ref{th:noisy-hessian}, the spectra of $\widetilde {\mathcal H}_0(y)$ and $\widetilde {\mathcal H}_\sigma(y)$ remain in a compact interval independent of $\sigma$. Moreover, by Corollary~\ref{co:noisy-hessian-gap}, their tangential and normal spectral clusters are separated by a uniform positive gap.
Hence one may choose a fixed positively oriented contour $\widetilde\Gamma$ that encloses the tangential cluster and excludes the normal cluster for both $\widetilde {\mathcal H}_0(y)$ and $\widetilde {\mathcal H}_\sigma(y)$.
Setting
\[
\Gamma:=\sigma^{-2}\widetilde\Gamma,
\]
we obtain
\[
\operatorname{length}(\Gamma)\le C_1\sigma^{-2},
\qquad
\operatorname{dist}\big(\Gamma,\operatorname{spec}(\mathcal H_0(y))\big)\ge C_2\sigma^{-2},
\qquad
\operatorname{dist}\big(\Gamma,\operatorname{spec}(\mathcal H_\sigma(y))\big)\ge C_3\sigma^{-2}.
\]
Therefore
\[
\sup_{z\in\Gamma}\|(\mathcal H_0(y)-zI)^{-1}\|_{\mathrm{op}}
+
\sup_{z\in\Gamma}\|(\mathcal H_\sigma(y)-zI)^{-1}\|_{\mathrm{op}}
\le C\sigma^2.
\]
For $\mathcal H$ with spectrum separated by $\Gamma$, define the corresponding Riesz projector by
\[
P[\mathcal H]=-\frac{1}{2\pi i}\int_{\Gamma}(\mathcal H-zI)^{-1}dz.
\]
Its derivative in the direction $K$ is
\[
DP[\mathcal H](K)=\frac{1}{2\pi i}\int_{\Gamma}(\mathcal H-zI)^{-1}K(\mathcal H-zI)^{-1}dz.
\]
Therefore,
\[
\nabla_u\widehat P_T(y)
=
DP[\mathcal H_\sigma(y)]\big(\nabla_u\mathcal H_\sigma(y)\big),
\qquad
\nabla_uP_T(y)
=
DP[\mathcal H_0(y)]\big(\nabla_u\mathcal H_0(y)\big).
\]
Set
\[
\Delta \mathcal H:=\mathcal H_\sigma(y)-\mathcal H_0(y),
\qquad
\Delta K:=\nabla_u\mathcal H_\sigma(y)-\nabla_u\mathcal H_0(y),
\qquad
K_0:=\nabla_u\mathcal H_0(y).
\]
Then
\[
\nabla_u\widehat P_T(y)-\nabla_uP_T(y)
=
DP[\mathcal H_\sigma(y)](\Delta K)
+
\big(DP[\mathcal H_\sigma(y)]-DP[\mathcal H_0(y)]\big)(K_0).
\]
Let $R_\sigma(z)=(\mathcal H_\sigma(y)-zI)^{-1}$ and
$R_0(z)=(\mathcal H_0(y)-zI)^{-1}$. By the above choice of $\Gamma$,
\[
\sup_{z\in\Gamma}\|R_\sigma(z)\|_{\mathrm{op}}
,\ 
\sup_{z\in\Gamma}\|R_0(z)\|_{\mathrm{op}}
\le C\sigma^2.
\]
Hence
\[
\|DP[\mathcal H_\sigma(y)](\Delta K)\|_{\mathrm{op}}
\le
C\operatorname{length}(\Gamma)
\sup_{z\in\Gamma}\|R_\sigma(z)\|_{\mathrm{op}}^2
\|\Delta K\|_{\mathrm{op}}
\le
C\sigma^2\|\Delta K\|_{\mathrm{op}}.
\]
Moreover, by the resolvent identity,
\[
R_\sigma(z)-R_0(z)
=
-R_\sigma(z)\Delta \mathcal HR_0(z),
\]
so that
\[
\sup_{z\in\Gamma}\|R_\sigma(z)-R_0(z)\|_{\mathrm{op}}
\le
C\sigma^4\|\Delta \mathcal H\|_{\mathrm{op}}.
\]
Therefore,
\begin{align*}
&\big\|
\big(DP[\mathcal H_\sigma(y)]-DP[\mathcal H_0(y)]\big)(K_0)
\big\|_{\mathrm{op}}\\
&\le
C\sup_{z\in\Gamma}(\|R_\sigma(z)\|+\|R_0(z)\|)\sup_{z\in\Gamma}\|R_\sigma(z)-R_0(z)\|_{\mathrm{op}}\|K_0\|_{\mathrm{op}}
\le
C\sigma^4\|\Delta \mathcal H\|_{\mathrm{op}}\|K_0\|_{\mathrm{op}}.
\end{align*}
We now estimate the three quantities above. First, by Theorem~\ref{th:noisy-hessian},
\[
\|\Delta \mathcal H\|_{\mathrm{op}}
=
\|\mathcal H_\sigma(y)-\mathcal H_0(y)\|_{\mathrm{op}}
=
O(1).
\]
Next,
\[
K_0
=
\nabla_u\mathcal H_0(y)
=
-\sigma^{-2}\Big(\nabla_uP_N(y)+\nabla_u \big((I_{T_xM}-A_y^{-1})P_T(y)\big)\Big),
\]
and since $\nabla P_T$, $\nabla P_N$, and $\nabla A_y^{-1}$ are uniformly bounded on $T(\tau-\varepsilon)$,
\[
\|K_0\|_{\mathrm{op}}\le C\sigma^{-2}.
\]
Finally, we claim that $\|\Delta K\|_{\mathrm{op}} \le C$. From the expansion in the proof of
Theorem~\ref{th:noisy-hessian},
\[
\mathcal H_\sigma(y)
=
-\sigma^{-2}P_N(y)
-\sigma^{-2}(I_{T_xM}-A_y^{-1})P_T(y)
-\frac12\nabla^2\log\det A_y
+\nabla^2R(y,\sigma).
\]
Thus
\[
\Delta \mathcal H 
=
-\frac12\nabla^2\log\det A_y
+\nabla^2R(y,\sigma).
\]
Differentiating in the tangential direction $u\in T_xM$, the above terms satisfy
\[
\|\nabla_u\nabla^2\log\det A_y\|_{\mathrm{op}} \le C,
\qquad
\|\nabla_u\nabla^2R(y,\sigma)\|_{\mathrm{op}} \le C\sigma,
\]
uniformly on $T(\tau-\varepsilon)$, by smoothness of $A_y$ and
Corollary~\ref{co:log-density}. Therefore,
\[
\|\Delta K\|_{\mathrm{op}}
=
\|\nabla_u\mathcal H_\sigma(y)-\nabla_u\mathcal H_0(y)\|_{\mathrm{op}}
\le C.
\]
Substituting the bounds for $\Delta \mathcal H$, $\Delta K$, and $K_0$ into the previous estimates, we obtain
\[
\|\nabla_u\widehat P_T(y)-\nabla_uP_T(y)\|_{\mathrm{op}}
\le C\sigma^2.
\]
Combining this with the estimate for the second term yields
\[
\|\widehat\Pi_y-\Pi_x\|_{\mathrm{op}}
\le
C(\|v_y\|+\sigma^2).
\]
Since $x=\pi(y)$, this proves the claim.
\end{proof}

\subsection{Proofs for Section~\ref{sec:metric-change}}\label{sec:pf-metric}
Here we prove the results in Section~\ref{sec:metric-change}.
\subsubsection{Proof of Theorem~\ref{th:metric-totally-umbilical-submanifold}}\label{sec:pf-metric totally umbilical submanifold}
\begin{proof}
Write $P:=P_\sigma$. Consider the conformal deformation
\[
g_P = e^{2f(P)}g_E,
\]
where $f:\mathbb{R}_+\to\mathbb{R}$ is to be determined. We write $\Pi^{E}$ and $\Pi^{P}$ for the second fundamental forms of $M$ with respect to $g_E$ and $g_P$, respectively, and $H$ for the mean curvature vector with respect to $g_E$.

It is standard that under a conformal change, the second fundamental form transforms as
\begin{equation}\label{eq:conformal-Pi}
\Pi^P_x(u,v)
=
\Pi^E_x(u,v)-g_E(u,v)\big(\nabla f(P(x))\big)^\perp,\qquad u,v\in T_xM,
\end{equation}
where $\nabla$ and $(\cdot)^\perp$ are taken with respect to the ambient Euclidean structure.
Since $M$ is totally umbilical, there is a mean curvature vector $H_x\in T_x^\perp M$ such that
\[
\Pi^E_x(u,v) = g_E(u,v)H_x.
\]
Substituting the umbilical form into \eqref{eq:conformal-Pi} gives
\[
\Pi^P_x(u,v)
=
g_E(u,v)\big(H_x - (\nabla f(P(x)))^\perp\big).
\]
The inclusion $(M,g)\hookrightarrow(\mathbb{R}^D,g_P)$ is totally geodesic if and only if $\Pi^P_x\equiv0$ for all $x\in M$, thus
\[
H_x = (\nabla f(P(x)))^\perp,\qquad x\in M.
\]
Since $f$ depends on $x$ only through $P(x)$, the chain rule yields
\[
\nabla f(P(x)) = f'(P(x))\nabla P(x) = f'(P(x))P(x)\nabla\log P(x).
\]
By Theorem~\ref{th:gradient-noisy}, at points $x\in M$ the estimator satisfies
\[
\nabla \log P(x)=\frac{d}{2}H_x + O(\sigma).
\]
Therefore, we obtain the scalar relation
\[
H_x
=
f'(P(x))P(x)\frac{d}{2}H_x + O(\sigma).
\]
If $H_x=0$, the leading term matching condition is automatic. Otherwise, matching the leading term requires
\[
\frac{d}{2}f'(P(x))P(x) = 1.
\]
Thus we choose 
\[
f(P)=\frac{2}{d}\log P + C.
\]
Taking $C=0$ yields
\[
f(P) = \frac{2}{d}\log P,
\qquad
g_P = e^{2f(P)}g_E = P^{4/d}g_E.
\]
With this choice,
\[
H_x-(\nabla f(P(x)))^\perp=O(\sigma),
\]
and hence
\[
\Pi^P_x(u,v)=O(\sigma)g_E(u,v).
\]
Therefore, for every $x\in M$,
\[
\|\Pi_x^{P}\|_{\mathrm{op}}\le C\sigma,
\]
and hence the inclusion $(M,g)\hookrightarrow (\mathbb{R}^D,g_P)$ is asymptotically totally geodesic as $\sigma\to 0$.
\end{proof}

\subsubsection{Proof of Theorem~\ref{th:metric-hypersurface}}\label{sec:pf-metric hypersurface}
\begin{proof}
Write $P:=P_\sigma$. Recall that
\[
g_P=d(\log P)\otimes d(\log P),
\qquad
(g_P)_{ij}=\partial_i\log P\partial_j\log P.
\]
Since $g_P$ is rank-one, it is degenerate and does not define a genuine Riemannian metric. We therefore do not invoke the Levi--Civita connection. Instead, using the Moore--Penrose pseudo-inverse
\[
(g_P)_{\mathrm{MP}}^{ij}
=
\frac{1}{\|\nabla\log P\|^4}
\partial_i\log P\partial_j\log P,
\]
we define the associated Christoffel-type coefficient field
\[
\widetilde \Gamma^m_{ij}
:=
(g_P)_{\mathrm{MP}}^{mk}
\widetilde \Gamma_{ij,k},
\qquad
\widetilde \Gamma_{ij,k}
:=
\frac12
\big(
\partial_j(g_P)_{ik}
+\partial_i(g_P)_{jk}
-\partial_k(g_P)_{ij}
\big).
\]
A direct computation gives
\begin{align*}
\widetilde \Gamma_{ij,k}
&=
\frac{1}{2}\big(
\partial_j\partial_i\log P\partial_k\log P
+\partial_i\log P\partial_j\partial_k\log P
+\partial_i\partial_j\log P\partial_k\log P
\\
&\hspace{1cm}
+\partial_j\log P\partial_i\partial_k\log P
-\partial_k\partial_i\log P\partial_j\log P
-\partial_i\log P\partial_k\partial_j\log P
\big)
\\
&=
(\partial_i\partial_j\log P)\partial_k\log P,
\end{align*}
and hence
\begin{align*}
\widetilde\Gamma^m_{ij}
&=
\frac{1}{\|\nabla\log P\|^4}
\partial_m\log P\partial_k\log P
(\partial_i\partial_j\log P)\partial_k\log P
\\&=
\frac{1}{\|\nabla\log P\|^2}
(\partial_i\partial_j\log P)\partial_m\log P.
\end{align*}
Accordingly, the induced acceleration equation of $\ddot\gamma^k(t)=- \widetilde \Gamma^k_{ij} \dot\gamma^i(t)\dot\gamma^j(t)$ is
\begin{equation}\label{eq:formal-autoparallel}
\ddot\gamma(t)
=
-\frac{1}{\|\nabla\log P(\gamma(t))\|^2}
\nabla^2\log P(\gamma(t))(\dot\gamma(t),\dot\gamma(t))
\nabla\log P(\gamma(t)).
\end{equation}
Now let $\gamma$ be a geodesic on $(M,g)$, viewed as a curve in $\mathbb{R}^D$. Its Euclidean acceleration is
\[
\ddot\gamma(t)=\Pi_{\gamma(t)}(\dot\gamma(t),\dot\gamma(t)).
\]
By Theorem~\ref{th:hypersurface}, in the hypersurface case,
\[
\Pi_{\gamma(t)}(\dot\gamma(t),\dot\gamma(t))
=
-\frac{1}{\|\nabla\log P(\gamma(t))\|^2}
\nabla^2\log P(\gamma(t))(\dot\gamma(t),\dot\gamma(t))
\nabla\log P(\gamma(t))
+O(\sigma),
\]
uniformly along $M$. Comparing with \eqref{eq:formal-autoparallel}, we conclude that $\gamma$ satisfies the induced acceleration equation up to an $O(\sigma)$ error. This proves the claim.
\end{proof}

\section*{Acknowledgments}J.C. is partly supported by the Shanghai Institute for Mathematics and Interdisciplinary Sciences (SIMIS). R.L. is a research fellow supported by the Singapore Ministry of Education Tier 2 grant A-8001562-00-00 at the National University of Singapore. Z.Y. has been supported by the Singapore Ministry of Education Tier 2 grant A-8001562-00-00 and the Tier 1 grant (A-8004146-00-00 and A-8002931-00-00) at the National University of Singapore.

\appendix

\section{Proof of Lemmas in Section~\ref{sec:prelim}}
\label{app:lemma-sec2}

\subsection{Proof of Lemma~\ref{lem:exponential-map-expansion}}\label{sec:pf-exponential-map-expansion}
\begin{proof}
     Let $x\in M$ and $\gamma_u\colon [0,1] \to M$ be the  geodesic in $M$ with initial condition $\gamma_u(0) = x$, $\gamma_u'(0) = u$. Then $\exp_{x}(u) = \gamma_u(1)$. Consider the Taylor expansion of $\exp_x$ around the origin of $T_xM$, it will be written as
    \[
    \exp_x(u) = x + I_x(u) + \frac{1}{2} Q_x(u) + \dots,
    \]
where $I_x$, $Q_x$ are respectively linear and bilinear map in $T_xM$. 
Therefore 
\[
\gamma_u(t) = x + I_x(u)t + \frac{1}{2} Q_x(u)t^2 +O(t^3). 
\]
Hence, $I_x(u) = \gamma_u'(0) = u$ and $ Q_x(u) = \gamma_u''(0)$. Since $\gamma$ is a geodesic, we have $\nabla^M_{\gamma_u'(0)} \gamma_u'(0) = 0$, which implies $\gamma_u''(0) = \Pi_x(u,u)$. Equivalently, 
\[
\exp_x(u) = x + u + \frac{1}{2}\Pi_x(u,u) + O(\|u\|^3).
\]
\end{proof}

\subsection{Proof of Lemma~\ref{lem:volume-element-expansion}}\label{sec:pf-volume element expansion}
\begin{proof}
Let $x\in M$, and work in exponential coordinates centered at $x$.
Since $g_{ij} | _x  = \delta_{ij}$, $D\exp_x$ is the identity map at the origin of $T_xM$, which implies $\frac{\partial}{\partial{u_i}} = D(\exp_x)_0(u_i) = u_i.$ The geodesic parametrization in the exponential coordinates is given by 
\[
c(t) = t(u^1,u^2,\dots,u^d).
\]
Putting this into the geodesic equation, we get for any $k$, 
\[
0 = \ddot{c^k(t)}  + \dot{c^i(t)}\dot{c^j(t)}\Gamma_{ij}^k(\exp (c(t))) = u^iu^j\Gamma_{ij}^k(\exp (c(t))).
\]
Let $t=0$, we have the quadratic form
\[
u^iu^j\Gamma_{ij}^k(x) = 0
\]
for any $u^iu^j$ and symmetric $\Gamma_{ij}^k$. Therefore $\Gamma^k_{ij}(0) = 0$ for all $i,j,k$, which implies $\partial_i g_{jk} + \partial_j g_{ik} - \partial_k g_{ij} = 0 $ at the origin. So $\partial_k g_{ij}|_x = 0$ follows from 
\[
\partial_k g_{ij} = \frac{1}{2}(\partial_k g_{ij} + \partial_j g_{ik} - \partial_i g_{jk}+ \partial_k g_{ij} + \partial_i g_{jk} - \partial_j g_{ik}) = 0.
\]
Hence $\sqrt{\operatorname{det} g_{ij}(x)} = 1$ and its first derivative vanish. Taking the Taylor expansion at $x$ we get 
\[
\sqrt{\operatorname{det} g_{ij}(x,u)} = 1 + O(\|u\|^2).
\]
Since $M$ is compact and the metric is smooth, the second derivatives of $g_{ij}$ in these coordinates are uniformly bounded in $x\in M$, so the $O(\|u\|^2)$ term is uniform in $x$.
\end{proof}

\section{Proof of Lemmas in Section~\ref{sec:density-and-derivatives}}
\label{app:lemma-sec3}
\subsection{Proof of Lemma~\ref{lem:derivative-of-v-pi}}\label{sec:pf-derivative of v pi}
\begin{proof}
Fix $y \in \mathcal T(\tau)$ and write $x := \pi(y) \in M$ and $v := v_y$. Let $S_v \colon T_xM \to T_xM$ be the shape operator associated with the normal vector $v$, so that 
\[
\langle S_v (u), w \rangle = \langle \Pi_x(u,w), v \rangle, \qquad S_v(u) = -(\nabla_u v)^\top,
\]
for $u,w\in T_xM$. Then we have $A_y = I_{T_xM} - \langle v_y,\Pi_{x}\rangle = I_{T_xM} - S_v$ on $T_xM$.

Let $\omega \in \mathbb{R}^D$ be arbitrary, and choose a smooth curve $\gamma \colon (-\delta,\delta) \to \mathcal T(\tau)$ such that $\gamma(0) = y$ and $\gamma'(0) = \omega$. Set
\[
\pi(t) := \pi(\gamma(t)) \in M, \qquad v(t) := v_{\gamma(t)} = \gamma(t) - \pi(t)
\in T_{\pi(t)}^\perp M.
\]
Then for any tangent vector field $Y(t) \in T_{\pi(t)}M$ along $\pi(t)$ we have
\[
\langle v(t), Y(t) \rangle = 0, \qquad \forall t.
\]
Differentiating at $t=0$ gives
\begin{align*}
0 
= \frac{d}{dt} \langle v(t), Y(t) \rangle \Big|_{t=0} 
= \langle v'(0), Y(0) \rangle
   + \langle v(0), \nabla_{\pi'(0)} Y \rangle.
\end{align*}
Here $v'(0) = \nabla_\omega v$ and $\pi'(0) = \nabla_\omega \pi(y) \in T_xM$.
Using the definition of the shape operator, we have
\[
\langle v(0), \nabla_{\pi'(0)} Y \rangle
= \langle S_v(\pi'(0)), Y(0) \rangle.
\]
Thus
\[
\langle v'(0), Y(0) \rangle = \langle (\nabla_\omega v_y)^\top, Y(0) \rangle
= - \langle S_v(\nabla_\omega \pi(y)), Y(0) \rangle
\]
for all $Y(0) \in T_xM$, and therefore
\begin{equation}\label{eq:lemma31-tangent}
(\nabla_\omega v_y)^\top = - S_v(\nabla_\omega \pi(y)).
\end{equation}
On the other hand, since $v_y = y - \pi(y)$, we also have
\begin{equation}\label{eq:lemma31-basic}
\nabla_\omega v_y = \omega - \nabla_\omega \pi(y).
\end{equation}
Taking the tangential component in \eqref{eq:lemma31-basic} yields
\[
(\nabla_\omega v_y)^\top = \omega^\top - \nabla_\omega \pi(y),
\]
because $\nabla_\omega \pi(y) \in T_xM$. Combining this with \eqref{eq:lemma31-tangent} we obtain
\[
\omega^\top - \nabla_\omega \pi(y) = - S_v(\nabla_\omega \pi(y)),
\]
or equivalently
\[
(I_{T_{\pi(y)}M} - S_v)\nabla_\omega \pi(y) = \omega^\top.
\]
By Lemma~\ref{lem:second-fundamental-form-bound} and the reach assumption, we have $\|S_v\|_{\mathrm{op}} \le \| \Pi_x\|_{\mathrm{op}} \|v\| \le \tau^{-1}\|v\| < 1$ for $y \in \mathcal T(\tau)$, hence $A_y = I_{T_xM}- S_v$ is invertible on $T_xM$. Therefore
\[
\nabla_\omega \pi(y) = A_y^{-1} \omega^\top.
\]
Substituting this into \eqref{eq:lemma31-basic} gives
\[
\nabla_\omega v_y = \omega - A_y^{-1}\omega^\top.
\]
\end{proof}

\section{Proof of Lemmas in Section~\ref{sec:compute-geometry}}\label{app:lemma-sec4}

\subsection{Proof of Lemma~\ref{lem:extended-normal-vector-field}}\label{sec:pf-extended normal vector field}
\begin{proof}
Write $x:=\pi(y)$. 
Choose a smooth curve $\gamma(t)\subset M$ such that $\gamma(0)=x$ and $\dot\gamma(0)=u$. By Lemma~\ref{lem:derivative-of-v-pi}, there exists a smooth curve $y(t)\subset \mathcal T(\tau)$ such that
\[
y(0)=y,\qquad \pi(y(t))=\gamma(t),\qquad \dot y(0)=A_yu.
\]
Let $Y(t)$ be a tangent vector field along $\gamma(t)$ such that $Y(t)\in T_{\gamma(t)}M$ and $Y(0)=v$. Then, by the definition of the second fundamental form we have $\Pi_x(u,v)=(\nabla_u Y(0))^\perp$.
%covariant derivative is metric
Since $\mathcal N(y(t))$ is normal to $M$ at $\gamma(t)=\pi(y(t))$, we have
\[
\langle \mathcal N(y(t)), Y(t)\rangle =0
\qquad\text{for all }t.
\]
Differentiating at $t=0$ gives
\[
0
=
\frac{d}{dt}\Big|_{t=0}\langle \mathcal N(y(t)),Y(t)\rangle
=
\big\langle \nabla_{A_yu}\mathcal N(y),v\big\rangle
+
\big\langle \mathcal N(y),\nabla_uY(0)\big\rangle.
\]
Since $\mathcal N(y)$ is normal and $Y(t)$ is tangent along $\gamma(t)$, only the normal component of $\nabla_uY(0)$ contributes, so
\[
\big\langle \mathcal N(y),\nabla_uY(0)\big\rangle
=
\big\langle \mathcal N(y),\Pi_x(u,v)\big\rangle.
\]
Hence
\[
\big\langle \nabla_{A_yu}\mathcal N(y),v\big\rangle
=
-
\big\langle \mathcal N(y),\Pi_x(u,v)\big\rangle.
\]
Since $T_x^\perp M$ is one-dimensional and spanned by $\mathcal N(y)$, we must have
\[
\Pi_x(u,v)
=
-\frac{1}{\|\mathcal N(y)\|^2}
\big\langle \nabla_{A_yu}\mathcal N(y),v\big\rangle\mathcal N(y).
\]
This proves the claim.
\end{proof}

\subsection{Proof of Lemma~\ref{lem:second-fundamental-form-general}}\label{sec:pf-second fundamental form_general}
\begin{proof}
Let $x\in M$, $u,v\in T_xM$ and $V$ be a $C^1$ vector field defined in a neighborhood of $x$ in $\mathcal T(\tau)$ such that $V(x)=v$ and $V$ is tangent to $M$ along $M$, i.e.,\ $V(z)\in T_zM$ for all $z\in M$. Then the vector field
\[
W(z) := P_T(z)V(z) - V(z)
\]
vanishes identically on $M$. In particular, $W(x)=0$. Differentiating along $u$ at $x$ and using the Levi--Civita connection in $\mathbb{R}^D$, we obtain
\[
0
=
\nabla_u W(x)
=
\nabla_u P_T(x)v
+
P_T(x)\nabla_u V(x)
-
\nabla_u V(x).
\]
Projecting this identity onto the normal space $T_x^\perp M$ gives
\[
(\nabla_u V(x))^\perp
=
(I_D-P_T(x))\nabla_u V(x)
=
(I_D-P_T(x))\nabla_uP_T(x)v,
\]
since $(I_D-P_T)P_T=0$. Next, differentiate the identity $P_T^2=P_T$ at $x$ in the direction $u$:
\[
\nabla_uP_T\cdot P_T + P_T\cdot\nabla_uP_T = \nabla_uP_T.
\]
Multiplying on the left by $P_T$ and using $P_T^2=P_T$ yields
\[
P_T\nabla_uP_TP_T = 0.
\]
For $v\in T_xM$ we have $v=P_Tv$, and hence
\[
P_T\nabla_uP_Tv = P_T\nabla_uP_TP_Tv = 0.
\]
Therefore $\nabla_uP_Tv$ is purely normal and
\[
(\nabla_u V(x))^\perp = \nabla_uP_T(x)v.
\]
Recalling that $\Pi_x(u,v)=(\nabla_u V(x))^\perp$ by definition, we conclude
\[
\Pi_{x}(u,v)
=
\nabla_uP_T(x)v.
\]
\end{proof}

\vskip 0.2in

\bibliographystyle{plainnat}
\bibliography{main}  

\end{document}